\documentclass[a4paper]{amsart}
\usepackage{amssymb} 
\usepackage[T1]{fontenc}
\usepackage[latin1]{inputenc}
\usepackage{amsfonts}
\usepackage{amsxtra}
\usepackage{ae}
\usepackage[all]{xy}
\usepackage{enumerate}
\usepackage{url}
\usepackage{verbatim}
\usepackage{pgf,tikz}
\usepackage{tikz-cd}
\usepackage{stmaryrd}

\usepackage{ragged2e}
\usepackage{amsmath}
\usepackage{amssymb}
\usepackage{mathtools}
\usepackage{tikz-cd}
\usepackage{bbm}
\usepackage{mathrsfs}
\usepackage{amsthm}
\usepackage{verbatim}
\usepackage{upgreek}
\usepackage{relsize}
\usepackage{enumitem}
\usepackage{dsfont}
\usepackage{graphicx}
\usepackage{amssymb}
\usepackage[all]{xy}
\usepackage{url}

\newcommand{\tp}{\!\!
{\scriptstyle
\text{
\raisebox{0.7pt}{
\textcircled{\raisebox{-1.3pt}{$\mathsf{T}$}}
} % \raisebox
} % \text
} % \scriptstyle
\!\!}

\usetikzlibrary{arrows,patterns}

\include{diagram}

\newcommand*{\ket}{\rangle}
\newcommand*{\bra}{\langle}

\newcommand*{\HH}{\mathcal{H}}

\newcommand*{\M}{\mathcal{M}}

\newcommand*{\T}{\mathcal{T}}

\newcommand*{\SSym}{{\bf Sym}}

\newcommand*{\Rep}{{\bf Rep}}
\newcommand*{\Corep}{{\bf Corep}}

\newcommand*{\red}{\mathsf{r}}
\newcommand*{\full}{\mathsf{f}}
\newcommand*{\Irr}{\mathsf{Irr}}

\renewcommand*{\Irr}{\mathsf{Irr}}
\newcommand*{\Poly}{\mathcal{O}}

\DeclareMathOperator{\Sym}{Sym}
\DeclareMathOperator{\Qut}{Qut}

\DeclareMathOperator{\rel}{rel}

\DeclareMathOperator{\Aut}{Aut}

\DeclareMathOperator{\ord}{ord}

\DeclareMathOperator{\id}{id}

\DeclareMathOperator{\supp}{supp}

\DeclareMathOperator{\Wr}{Wr}
\DeclareMathOperator{\wwr}{wr}

\newenvironment{bnum}
{\begin{list}{}
    {\setlength{\labelwidth}{15pt}
     \setlength{\leftmargin}{\labelwidth}
    }
}
{\end{list}}

\numberwithin{equation}{section}
\theoremstyle{change}
\newtheorem{theorem}{Theorem}[section]
\newtheorem{prop}[theorem]{Proposition}
\newtheorem{lemma}[theorem]{Lemma}

\newtheorem{definition}[theorem]{Definition}
\newtheorem{question}[theorem]{Question}

\begin{document}

\title{Infinite quantum permutations}

\author{Christian Voigt}
\address{School of Mathematics and Statistics \\
         University of Glasgow \\
         University Place \\
         Glasgow G12 8QQ \\
         United Kingdom 
}


%\keywords{Quantum automorphisms, quantum groups, infinite graphs}

\thanks{This work was supported by EPSRC grant EP/T03064X/1.}

\maketitle

\begin{abstract}
We define and study quantum permutations of infinite sets. This leads to discrete quantum groups which can be viewed as infinite variants of the quantum permutation groups 
introduced by Wang. More precisely, the resulting quantum groups encode universal quantum symmetries of the underlying sets among all discrete quantum groups. 

We also discuss quantum automorphisms of infinite graphs, including some examples and open problems regarding both the existence and non-existence of quantum symmetries in 
this setting.  
\end{abstract}

\section{Introduction} 

The quantum permutation group $ S_n^+ $, introduced by Wang \cite{Wangqsymmetry}, is the universal compact quantum group acting on $ n $ points. This quantum group 
has been studied extensively from various perspectives, with motivation coming from operator algebras, subfactors, and free probability, see for 
instance \cite{BBCsurvey}, \cite{BANICA_quantumpermutationgroups}, \cite{Brannanquantumautomorphism}, \cite{KOESTLER_SPEICHER_definetti}. 
Building on the construction of quantum permutation groups, Banica and Bichon introduced quantum automorphism groups of finite 
graphs \cite{BICHON_qutgraphs}, \cite{Banicaqutgraph}, \cite{Banicasmallmetric}. This, in turn, has led to the discovery of interesting links between quantum groups, 
graph theory, and the theory of non-local games \cite{LUPINI_MANCINSKA_ROBERSON_nonlocalgames}, \cite{MANCINSKA_ROBERSON_quantumisoplanar}. 

In view of these developments, it is natural to ask for infinite versions of quantum permutation groups, that is, quantum generalisations of the symmetric group of an infinite set. 
Goswami and Skalski \cite{GOSWAMI_SKALSKI_quantumpermutations} addressed this question by introducing two quantum semigroups of infinite quantum permutations. The first one is a quantum 
analogue of the group of permutations moving only finitely many points, and can be viewed as a certain inductive limit of the quantum permutation groups $ S_n^+ $ 
for $ n \in \mathbb{N} $. Since $ S_n^+ $ is not discrete one has to be careful to give meaning to such a limit, and Goswami and Skalski do this by working on the level 
of von Neumann algebras. Their second construction yields the universal von Neumann algebra generated by the entries of an infinite magic unitary matrix, which can be 
viewed as a quantum analogue of the group of all permutations of an infinite set. It is unclear, however, if either of these objects fit into the theory of locally compact quantum 
groups in the sense of Kustermans and Vaes \cite{KVLCQG}. 

In this paper we propose a slightly different approach to infinite quantum permutations which allows one to obtain genuine quantum groups. As 
in \cite{GOSWAMI_SKALSKI_quantumpermutations}, the key ingredient is the $ * $-algebra generated by the entries of an infinite magic unitary matrix, but in contrast we single 
out different classes of its representations. This is inspired by the theory of non-local games and their associated game algebras \cite{HMPSsynchronousgames}. 
Winning strategies for a synchronous game can be encoded by different types of representations of the game algebra, thus emphasizing the role of representation theory for 
this algebra. In particular, finite dimensional $ * $-representations correspond to winning quantum strategies, and in the case of the graph isomorphism game these are closely 
related with certain finite dimensional $ * $-representations of the function algebras of the quantum permutation groups $ S_n^+ $. Let us point out that studying the structure 
of such representations amounts essentially to understanding matrix models for $ S_n^+ $, see \cite{BANICA_quantumpermutationgroups}. 

Building on these observations we define a quantum version $ \Sym^+(X) $ of the full symmetric group $ \Sym(X) $ of an arbitrary set $ X $, and also a quantum 
version $ \Sigma^+(X) $ of the subgroup $ \Sigma(X) \subset \Sym(X) $ consisting of permutations which move only finitely many points. Both $ \Sym^+(X) $ and $ \Sigma^+(X) $ 
are discrete quantum groups, and if $ X = \{1,\dots, n\} $ is finite they can be viewed as the discretisation of the compact quantum group $ S_n^+ $. 
Here by discretisation we mean the notion dual to quantum Bohr compactification defined and studied by So{\l}tan \cite{SOLTAN_compactifications}, \cite{SOLTAN_bohr}. 

In a similar way we define quantum automorphisms associated to arbitrary simple graphs. For a finite graph $ X $, the resulting quantum group $ \Qut_\delta(X) $ 
can again be viewed as the discretisation of the corresponding compact quantum automorphism groups $ \Qut(X) $, and this allows one to transfer a number of results and 
techniques to the infinite case. We shall illustrate 
this by looking at some examples, largely building on the work of Schmidt \cite{SCHMIDT_foldedcube}, \cite{SCHMIDT_distancetransitive}. At the same time, we list a few open problems 
whose resolution we expect to be helpful for gaining a better understanding of the genuinely infinite aspects of the theory. 

After the first version of this paper appeared, Rollier and Vaes published a very interesting construction of locally compact quantum automorphism 
groups for connected locally finite graphs \cite{ROLLIER_VAES_qut}. To such a graph one can naturally associate a multiplier Hopf $ * $-algebra 
in the sense of van Daele \cite{vDmult}, since in the connected locally finite case the relations for an infinite magic unitary compatible with the adjacency relations 
can be interpreted purely algebraically. The key result of \cite{ROLLIER_VAES_qut} is that this multiplier Hopf $ * $-algebra admits Haar weights. 
In order to construct these weights, Rollier-Vaes study a certain unitary tensor category associated to the graph, extending work by 
Man\v{c}inska-Roberson \cite{MANCINSKA_ROBERSON_quantumisoplanar}. As already noted in \cite{ROLLIER_VAES_qut}, the discretisation of 
the locally compact quantum group $ \Qut(X) $ defined by Rollier-Vaes identifies with the quantum group $ \Qut_\delta(X) $ which we consider here. 

Let us explain how the paper is organised. In section \ref{secprelim} we collect some preliminaries regarding quantum groups and fix our notation. 
Section \ref{secinfiniteqquantumpermutations} contains the definition of infinite quantum permutations and a description of the associated $ C^* $-tensor categories. 
In section \ref{secinfiniteqpg} we focus our attention on finite dimensional quantum permutations and the corresponding discrete quantum groups. 
We show that these quantum groups can be interpreted as universal quantum symmetry groups, in analogy to the considerations in \cite{Wangqsymmetry}. In 
section \ref{seclocallyfinite} we specialise to quantum permutations moving only finitely many points. We show that the resulting quantum groups are 
non-amenable as soon as the underlying set contains at least four elements. 
Section \ref{secfreewreath} contains an infinite version of the free wreath product construction first studied by Bichon \cite{Bichonfreewreathproducts}. 
More precisely we define unrestricted and restricted free wreath products of discrete quantum groups with respect to our infinite quantum permutation groups. 
Finally, in section \ref{secinfinitegraphs} we extend the discussion to the case of 
graphs and consider some examples. In particular, we show that the infinite Johnson graph $ J(\infty, 2) $ has no quantum symmetry. 
In constrast, graphs with disjoint automorphisms, the unit distance graph of $ \mathbb{R} $, and all infinite Hamming graphs have quantum symmetry. It is also shown that 
the Rado graph does not admit any non-classical quantum automorphisms of finite dimension. 

We conclude with some remarks on our notation. If $ \HH $ is a Hilbert space we write $ B(\HH) $ for the algebra of bounded operators on $ \HH $, and denote by $ [X] $ 
the closed linear span of a subset $ X $ of a Banach space. 
If $ A $ is a $ C^* $-algebra we write $ \Rep(A) $ for the $ C^* $-category of all nondegenerate finite dimensional $ * $-representations of $ A $. 
The multiplier algebra of $ A $ is denote by $ M(A) $. 

I would like to thank Matthew Daws and Stefaan Vaes for helpful discussions about infinite quantum permutations.

\section{Preliminaries} \label{secprelim}

In this section we collect some definitions and results from the theory of quantum groups, mainly in order to fix terminology and notation.  We refer 
to \cite{KVLCQG}, \cite{Wangqsymmetry}, \cite{Banicageneric}, \cite{Banicafusscatalan}, \cite{NTlecturenotes} for more background. 

Recall from \cite{KVLCQG} that a locally compact quantum group $ G $ is given by a Hopf $ C^* $-algebra $ C_0^\red(G) $ together with faithful left and right 
Haar weights. We write $ L^2(G) $ for the GNS-construction of the left Haar weight, so that we have $ C_0^\red(G) \subset B(L^2(G)) $ in a natural way. 
By definition, the right leg of the fundamental multiplicative unitary $ W \in B(L^2(G) \otimes L^2(G)) $ is the reduced group $ C^* $-algebra $ C^*_\red(G) $ 
of $ G $, and we have $ W \in M(C_0^\red(G) \otimes C^*_\red(G)) $. Both $ C_0^\red(G) $ and $ C^*_\red(G) $ admit full versions $ C_0^\full(G) $ and $ C^*_\full(G) $, 
in analogy to the full and reduced group $ C^* $-algebras of classical locally compact groups. 
The Pontrjagin dual $ \hat{G} $ of $ G $ is the locally compact quantum group determined by $ C_0^\red(\hat{G}) = C^*_\red(G) $ as Hopf $ C^* $-algebras. 

A locally compact quantum group $ G $ is called compact if $ C^\full_0(G) $ is unital, in which case we write $ C^\full(G) $ and $ C^\red(G) $ 
for the associated full and reduced algebras of functions. A compact quantum group $ G $ can be equivalently described by the Hopf $ * $-algebra $ \Poly(G) \subset C^\full(G) $ 
of representative functions. A locally compact quantum group $ G $ is discrete if its dual $ \hat{G} $ is compact. 
The full Hopf $ C^* $-algebra of functions on a discrete quantum group agrees with its reduced version, and is given by a $ C^* $-direct sum of finite dimensional matrix algebras. 
The matrix blocks appearing in this decomposition correspond to the irreducible corepresentations of the discrete quantum group, or 
equivalently, to the irreducible representations of its compact dual. 

If $ G $ is a discrete quantum group we write $ \Corep(G) = \Rep(\hat{G}) $ for the $ C^* $-tensor category of finite dimensional representations of $ \hat{G} $, 
and denote by $ \dim(t) $ the dimension of the Hilbert space underlying $ t \in \Corep(G) $. The 
category $ \Corep(G) $ is rigid in the sense that every object is dualisable, that is, every $ X \in \Corep(G) $ admits a dual object $ \overline{X} \in \Corep(G) $ together with 
morphisms $ ev_X: \overline{X} \otimes X \rightarrow 1, db_X: 1 \rightarrow X \otimes \overline{X} $ 
and $ ev_{\overline{X}}: X \otimes \overline{X} \rightarrow 1, db_{\overline{X}}: 1 \rightarrow \overline{X} \otimes X $ satisfying the so-called zig zag equations. 
Here $ 1 \in \Corep(G) $ is the tensor unit. We denote by $ \Irr(\hat{G}) $ the set of equivalence classes of irreducible objects in $ \Corep(G) = \Rep(\hat{G}) $, and say 
that $ G $ is countable if $ \Irr(\hat{G}) $ is a countable set. The category $ \Corep(G) $ admits a tautological fiber functor into the category of Hilbert spaces, and conversely, 
every rigid $ C^* $-tensor category $ \T $ together with a fiber functor determines a discrete quantum group $ G $ such that $ \T \simeq \Corep(G) $. This is known as the 
Tannaka-Krein reconstruction theorem \cite{Woronowiczsuqn}. 

A \emph{morphism} $ \iota: G \rightarrow H $ of locally compact quantum groups is a nondegenerate $ * $-homomorphism $ \iota^*: C_0^\full(H) \rightarrow M(C_0^\full(G)) $ 
which is compatible with the comultiplications. Every morphism can equivalently be described by a nondegenerate $ * $-homomorphism $ \iota_*: C^*_\full(G) \rightarrow M(C^*_\full(H)) $, 
again compatible with the comultiplications. A \emph{quantum subgroup} of a discrete quantum group $ G $ is given by a full $ C^* $-tensor subcategory of the 
category $ \Corep(G) $ of corepresentations of $ G $. The \emph{direct union} $ \bigcup_{i \in I} G_i $ of a family of discrete quantum groups over a directed set $ I $, 
together with inclusion morphisms $ G_i \rightarrow G_j $ for $ i \leq j $, is defined as the discrete quantum group corresponding to the direct limit of the 
corresponding $ C^* $-tensor categories. 

We shall say that a locally compact quantum group $ G $ is \emph{strongly amenable} if the canonical quotient map $ \lambda: C^*_\full(G) \rightarrow C^*_\red(G) $ is an isomorphism. 
It is called \emph{coamenable} if $ \hat{\lambda}: C_0^\full(G) \rightarrow C_0^\red(G) $ is an isomorphism \cite{BMTcqgcoamenability}. In either case 
we simply write $ C^*(G) $ instead of $ C^*_\full(G) \cong C^*_\red(G) $ or $ C_0(G) $ instead of $ C_0^\full(G) \cong C_0^\red(G) $, respectively. 
Every classical locally compact group is coamenable, and strongly amenable iff it is amenable. The same is true for discrete quantum groups \cite{Tomatsuamenablediscrete}. 
A discrete quantum group $ G $ is amenable iff the counit of $ \Poly(\hat{G}) $ extends to a bounded $ * $-homomorphism $ \epsilon: C^*_\red(G) \rightarrow \mathbb{C} $.  
It follows that if $ G = \bigcup_{i \in I} G_i $ is the direct union of a directed family of quantum subgroups $ G_i \subset G $ then the discrete quantum group $ G $ is 
amenable iff all $ G_i $ are amenable. In particular, a discrete quantum group is amenable iff all its countable quantum subgroups are amenable. 

Let $ G $ be a locally compact quantum group and $ B $ a $ C^* $-algebra. An \emph{action} of $ G $ on $ B $ is an 
injective $ * $-homomorphism $ \beta: B \rightarrow M(C_0^\red(G) \otimes B) $ such that $ (\Delta \otimes \id)\beta = (\id \otimes \beta) \beta $ 
and $ [\beta(B) (C_0^\red(G) \otimes 1)] = C_0^\red(G) \otimes B $. If $ G $ is a classical locally compact group this is equivalent to a strongly continuous action 
of $ G $ on $ B $ by $ * $-automorphisms. 

Next we review the definition of quantum permutation groups in the sense of Wang \cite{Wangqsymmetry}. By construction, the quantum permutation group $ S_n^+ $ is the quantum 
automorphism group of $ A = \mathbb{C}^n $, and we have the following explicit description. 

\begin{definition} 
Let $ n \in \mathbb{N} $. The quantum permutation group $ S_n^+ $ is the compact quantum group given by the universal $ C^* $-algebra generated by 
the entries of a magic unitary $ n \times n $-matrix $ u = (u_{ij}) $, that is, $ C^\full(S_n^+) $ the universal $ C^* $-algebra generated by projections $ u_{ij} $ 
for $ 1 \leq i,j \leq n $ such that 
\begin{equation*}
\sum_{i = 1}^n u_{ik} = 1, \qquad \sum_{j = 1}^n u_{kj} = 1 
\end{equation*}
for all $ 1 \leq k \leq n $. The comultiplication $ \Delta: C^\full(S_n^+) \rightarrow C^\full(S_n^+) \otimes C^\full(S_n^+) $ is defined by 
$ \Delta(u_{ij}) = \sum_{k = 1}^n u_{ik} \otimes u_{kj} $ on the generators. 
\end{definition} 

One obtains a canonical morphism of quantum groups $ S_n \rightarrow S_n^+ $, that is, a unital $ * $-homomorphism $ C^\full(S_n^+) \rightarrow C(S_n) $ compatible 
with comultiplications, where $ S_n $ is the symmetric group on $ n $ elements. In fact, $ C(S_n) $ is the abelianisation of $ C^\full(S_n^+) $. 

The structure of the quantum permutation group $ S_n^+ $ is well-understood for small values of $ n $, compare \cite{Banicasmallmetric}. In particular, 
for $ n = 1,2,3 $ the canonical morphism $ S_n \rightarrow S_n^+ $ is an isomorphism. For $ n = 4 $ the morphism $ S_n \rightarrow S_n^+ $ is 
no longer an isomorphism, and the $ C^* $-algebra $ C^\full(S_4^+) $ is infinite dimensional. While the quantum group $ S_4^+ $ is still coamenable, this is not the case 
for $ S_n^+ $ if $ n \geq 5 $. 

Let us also review the definition of quantum automorphism groups of finite graphs, see \cite{Banicasmallmetric}, \cite{Banicaqutgraph}. Here by a finite 
graph $ X = (V_X, E_X) $ we mean an undirected simple graph without loops, given by a finite set $ V_X $ of vertices and a set $ E_X \subset V_X \times V_X $ of edges such 
that $ (v,w) \in E_X $ iff $ (w,v) \in E_X $, and $ (v,v) \notin E_X $ for all $ v \in V_X $. 
The adjacency matrix of $ X = (V_X, E_X) $ is the matrix $ A_X \in M_{V_X}(\{0,1\}) $ determined by 
$$
(A_X)_{x,y} = 1 \Leftrightarrow (x,y) \in E_X, 
$$
and it determines the graph uniquely. Note that $ A_X $ can be viewed as a linear operator $ l^2(V_X) \rightarrow l^2(V_X) $. 

\begin{definition} \label{defqut}
Let $ X = (V_X, E_X) $ be a finite graph with adjacency matrix $ A_X $. The quantum automorphism group $ \Qut(X) $ of $ X $ is given by the 
universal $ C^* $-algebra $ C^\full(\Qut(X)) $ generated by elements $ u_{xy} $ for $ x,y \in V_X $ such that $ u = (u_{xy}) $ is a magic unitary matrix satisfying 
$$
u A_X = A_X u.  
$$
The comultiplication is given by $ \Delta(u_{xy}) = \sum_{z \in V_X} u_{xz} \otimes u_{zy} $ on the generators. 
\end{definition} 

By construction, the quantum automorphism group $ \Qut(X) $ is a quantum subgroup of $ S_n^+ $ for $ n = |V_X| $. The defining relation $ u A_X = A_X u $ can be equivalently 
expressed as saying  
$$
u_{vw} u_{xy} = 0 \text{ if } \rel(v,x) \neq \rel(w,y),  
$$
where the function $ \rel $ encodes the adjacency relation between vertices, which takes one of the values \emph{equal}, or \emph{adjacent}, or \emph{distinct and non-adjacent}.

\section{Infinite quantum permutations} \label{secinfiniteqquantumpermutations}

Throughout this section we fix a set $ X $, and we write $ \Sym(X) $ for the group of all permutations of $ X $.  
In the sequel we will mostly be interested in the case that $ X $ is countable, but the constructions work in general. 
 
The starting point of our discussion is the following definition, which is more or less implicit in the literature on quantum automorphisms in the case 
that $ X $ is finite. 

\begin{definition} \label{defquantumpermutation}
Let $ X $ be a set. A quantum permutation of $ X $ is a pair $ \sigma = (\HH_\sigma, u^\sigma) $ consisting of a Hilbert space $ \HH_\sigma $ 
and a family $ u^\sigma = (u^\sigma_{xy})_{x,y \in X} $ of projections $ u^\sigma_{xy} \in B(\HH_\sigma) $ such that 
\begin{bnum} 
\item[$\bullet$] For every $ x \in X $ the projections $ u^\sigma_{xz} $ for $ z \in X $ are pairwise orthogonal. 
\item[$\bullet$] For every $ y \in X $ the projections $ u^\sigma_{zy} $ for $ z \in X $ are pairwise orthogonal.
\item[$\bullet$] We have 
$$ 
\sum_{z \in X} u^\sigma_{xz} = 1 = \sum_{z \in X} u^\sigma_{zy} 
$$
for all $ x,y \in X $, with convergence understood in the strong operator topology. 
\end{bnum}
If $ \sigma = (\HH_\sigma, u^\sigma) $ and $ \tau = (\HH_\tau, u^\tau) $ are quantum permutations of $ X $ then an intertwiner from $ \sigma $ to $ \tau $ is a 
bounded linear operator $ T: \HH_\sigma \rightarrow \HH_\tau $ such that $ T u^\sigma_{xy} = u^\tau_{xy} T $ for all $ x, y \in X $. 
\end{definition} 

Note that the convergence of the infinite sums in Definition \ref{defquantumpermutation} can be interpreted equivalently in any of 
the weak, strong, strong*, $ \sigma $-weak, $ \sigma $-strong or $ \sigma $-strong* topologies. 
In the sequel all infinite sums of families $ (p_i)_{i \in I} $ of pairwise orthogonal projections will be understood this way, and if $ \sum_{i \in I} p_i = 1 $ then we also say
that $ (p_i)_{i \in I} $ is a \emph{partition of unity}. We can thus rephrase Definition \ref{defquantumpermutation} as saying that a 
quantum permutation of a set $ X $ is a matrix of projections indexed by $ X $ such that all rows and columns form partitions of unity. 
Sometimes we shall also refer to such a quantum permutation as a \emph{magic unitary} indexed by $ X $. It is not hard to see that 
the first two conditions for a quantum permutation in Definition \ref{defquantumpermutation} are in fact a consequence of the third. 
%Fix $ x,y $ and $ \xi \in \HH $. Then $ u_{xy} \xi = u_{xy} 1 u_{xy} \xi = \sum_z u_{xy} u_{xz} u_{xy} \xi $ and hence $ \sum_{z \neq y} \bra \xi u_{xy} u_{xz} u_{xy} \xi = 0 $. 
%This also means $ \bra \xi, \sum_{z \neq x} \bra \xi u_{xy} u_{xz} u_{xy} \xi \ket = 0 $. Each term in this sum is a non-negative scalar, 
%so that $ \bra \xi, u_{xy} u_{xz} u_{xy} \xi \ket = 0 $ for $ z \neq y $. Since $ u_{xy} u_{xz} u_{xy} $ is positive it follows 
%that $ u_{xy} u_{xz} u_{xy} = (u_{xz} u_{xy})^* u_{xz} u_{xy} = 0 $. Hence $ u_{xz} u_{xy} = 0 $ as required. 

By the \emph{dimension} of a quantum permutation $ \sigma = (\HH_\sigma, u^\sigma) $ we mean the dimension of its underlying Hilbert space. 
We note that a finite dimensional quantum permutation is row- and column-finite in the sense that for all $ x,y \in X $ we have $ u^\sigma_{xz} = 0 $ 
and $ u^\sigma_{zy} = 0 $ for all but finitely many $ z $. 

We say that two quantum permutations $ \sigma = (\HH_\sigma, u^\sigma) $ and $ \tau = (\HH_\tau, u^\tau) $ are \emph{unitarily equivalent} if there exists a unitary 
intertwiner between them. A quantum permutation $ \sigma $ is called \emph{irreducible} if the only intertwiners from $ \sigma $ to itself are multiples of the identity.  

\begin{lemma} \label{onedim}
Unitary equivalence classes of one-dimensional quantum permutations of a set $ X $ correspond bijectively to permutations of $ X $.  
\end{lemma} 

\begin{proof} 
This is almost immediate from the definitions. If $ \sigma \in \Sym(X) $ is a permutation then we obtain a quantum permutation $ Q_\sigma = (\mathbb{C}, u^\sigma) $ 
by setting $ u^\sigma_{xy} = \delta_{x \sigma(y)} $. Conversely, let $ \sigma = (\HH_\sigma, u^\sigma) $ be a quantum permutation of $ X $ of dimension one. Since the 
only projections in $ B(\HH_\sigma) = \mathbb{C} $ are $ 0 $ and $ 1 $, there exists for each $ y \in X $ a uniquely determined element $ C_\sigma(y) \in X $ such 
that $ u^\sigma_{C_\sigma(y) y} = 1 $, and this defines a permutation $ C_\sigma \in \Sym(X) $. These assignments yield mutually inverse bijections as claimed. 
\end{proof} 

By slight abuse of terminology, it is sometimes convenient not to distinguish between quantum permutations and their unitary equivalence classes and refer to the latter as 
quantum permutations as well. With this understood, the \emph{trivial quantum permutation} of $ X $ is the quantum permutation corresponding to the identity permutation 
in Lemma \ref{onedim}. 

Let us now discuss some standard procedures for constructing new quantum permutations out of given ones. 

\begin{definition} \label{deftensorstructure}
Let $ X $ be a set. 
\begin{bnum} 
\item[$\bullet$] The direct sum of quantum permutations $ \sigma = (\HH_\sigma, u^\sigma) $ and $ \tau = (\HH_\tau, u^\tau) $ of $ X $ is defined 
by $ \sigma \oplus \tau = (\HH_\sigma \oplus \HH_\tau, u^\sigma \oplus u^\tau) $,  where $ (u^\sigma \oplus u^\tau)_{xy} = u^\sigma_{xy} \oplus u^\tau_{xy} $ 
for all $ x,y \in X $. 

\item[$\bullet$] The tensor product of quantum permutations $ \sigma = (\HH_\sigma, u^\sigma) $ and $ \tau = (\HH_\tau, u^\tau) $ is defined 
by $ \sigma \otimes \tau = (\HH_\sigma \otimes \HH_\tau, u^\sigma \tp u^\tau) $ 
where $ (u^\sigma \tp u^\tau)_{xy} = \sum_{z \in X} u^\sigma_{xz} \otimes u^\tau_{zy} $ for all $ x,y \in X $. 

\item[$\bullet$] The contragredient $ \overline{\sigma} = (\HH_{\overline{\sigma}}, u^{\overline{\sigma}}) $ of a quantum permutation $ \sigma = (\HH_\sigma, u^\sigma) $ 
is defined by taking $ \HH_{\overline{\sigma}} $ to be the conjugate Hilbert space of $ \HH_\sigma $ and the family of 
projections $ u^{\overline{\sigma}} = (u^{\overline{\sigma}}_{xy}) $ determined by $ u^{\overline{\sigma}}_{xy}(\overline{\xi}) = \overline{u^\sigma_{yx}(\xi)} $ 
for $ \xi \in \HH_\sigma $. 
\end{bnum} 
\end{definition} 

It is straightforward to check that all operations listed in Definition \ref{deftensorstructure} are compatible with intertwiners in a natural way and yield indeed quantum permutations. 
This leads us immediately to the following observation. 

\begin{lemma} 
Let $ X $ be a set. The collection of all quantum permutations of $ X $ and their intertwiners forms naturally a $ C^* $-tensor category. 
\end{lemma} 

As we will discuss in more detail further below, one obtains a basic supply of quantum permutations of an infinite set by combining 
representations of $ C(S_n^+) $ for some $ n \in \mathbb{N} $ with classical permutations. 

Let us describe a different source of infinite quantum permutations. By definition, a \emph{partial quantum permutation} $ \sigma = (\HH, A, B, u) $ of
a set $ X $ consists of a Hilbert space $ \HH $, subsets $ A, B \subset X $ and projections $ u_{xy} \in B(\HH) $ for $ (x,y) \in A \times X \cup X \times B $ such that 
\begin{bnum} 
\item[$\bullet$] For every $ x \in X $ the projections $ u_{xz} $ are pairwise orthogonal whenever they are defined, 
\item[$\bullet$] For every $ y \in X $ the projections $ u_{zy} $ are pairwise orthogonal whenever they are defined,
\item[$\bullet$] For all $ x \in A $ and $ y \in B $ we have $ \sum_{z \in X} u_{xz} = 1 = \sum_{z \in X} u_{zy} $. 
\end{bnum} 
We call $ A $ the \emph{domain} and $ B $ the \emph{range} of the partial quantum permutation. The collection of all partial quantum permutations on a fixed Hilbert space is partially 
ordered by saying that $ \sigma \leq \tau $ if the domain and range of $ \sigma $ are contained in the domain and range of $ \tau $, respectively, and the operators $ u^\sigma_{xy} $ 
and $ u^\tau_{xy} $ agree whenever the former are defined. Of course, a partial quantum permutation with $ A = X = B $ is nothing but a quantum permutation in the sense of 
Definition \ref{defquantumpermutation}.  

Let $ X $ be a countable set and write $ e_x $ for the canonical basis vector of $ \ell^2(X) $ associated with $ x \in X $. We define the \emph{support} of a 
projection $ p \in B(\ell^2(X)) $ as the set of all elements $ x \in X $ such that $ p e_x $ is nonzero. 
Moreover we say that a partial quantum permutation $ \sigma = (\ell^2(X), A, B, u) $ of $ X $ has \emph{locally finite rank} if 
\begin{bnum} 
\item[$\bullet$] for every $ x \in A $ and $ v \in X $ there are only finitely many elements $ y \in X $ such that $ v $ is 
contained in the support of $ u_{xy} $,  
\item[$\bullet$] for every $ y \in B $ and $ v \in X $ there are only finitely many elements $ x \in X $ such that $ v $ is contained in the support of $ u_{xy} $, 
\item[$\bullet$] the support of the projection $ u_{xy} $ is finite for all $ (x,y) \in A \times X \cup X \times B $. 
\end{bnum} 
In order to construct quantum permutations of $ X $ we can now use a variant of the back and forth method, compare for instance section 2.4 in \cite{MARKER_modeltheory}. 

Fix an enumeration $ X = \{x_1, x_2, x_3, \dots \} $. In a first step we set $ A_1 = \{x_1\} $ and $ B_0 = \emptyset $. Moreover we let $ (u_{x_1 y})_{y \in X} $ be a 
partition of unity in $ B(l^2(X)) $ consisting of finitely supported projections, such that each element of $ X $ is contained in the support of only finitely 
many $ u_{x_1 y} $. 
For instance, we can take $ u_{x_1 y} $ to be the orthogonal projection corresponding to the canonical basis vector $ e_y \in \ell^2(X) $. We obtain a partial quantum 
permutation $ \sigma_{1,0} = (\ell^2(X), A_1, B_0, u) $ this way. 

Now assume that we have constructed a partial quantum permutation $ \sigma_{m,n} = (\ell^2(X), A_m, B_n, u) $ of $ X $ of locally finite rank 
with domain $ A_m =\{x_1, \dots, x_m\} $ and range $ B_n = \{x_1, \dots, x_n\} $. We can then extend the domain of $ \sigma $ by adding $ x_{m + 1} $ to $ A_m $. 
More precisely, let $ p = \sum_{j = 1}^n u_{x_{m + 1} x_j} $, 
and consider an arbitrary family $ (p_k)_{k \in \mathbb{N}} $ of finitely supported projections in $ l^2(X) $ such 
that each element of $ X $ is contained in the support of only finitely many $ p_k $ and $ \sum_{k \in \mathbb{N}} p_k = 1 - p $. 
By construction of $ \sigma_{m,n} $, the support of $ p_k $ intersects nontrivially with the support of only finitely many of the existing projections $ u_{xy} $. 
Hence we find $ y \in X $ such that $ p_k $ is orthogonal to all projections $ u_{x_i y} $ for $ 1 \leq i \leq m $,  and we set $ u_{x_{m + 1} y} = p_k $. This can be done for 
all $ k \in \mathbb{N} $, and we let the remaining operators $ u_{x_{m + 1} y} $ be zero. By construction this yields a partial quantum permutation $ \sigma_{m + 1, n} $ 
of $ X $ of locally finite rank with domain $ A_{m + 1} = A_m \cup \{x_{m + 1}\} $ and range $ B_n $. 

In a similar way we can extend the range of a partial quantum permutation $ \sigma_{m,n} $ of locally finite rank as above. Taking unions we see that there exists a quantum 
permutation $ \sigma = (\ell^2(X), u) $ of $ X $ of locally finite rank which restricts to the partial quantum permutations $ \sigma_{m,n} $ obtained along the way 
for all $ m, n \in \mathbb{N} $. In particular, every point of $ X $ is contained in the domain and range of $ \sigma $ as required. 

We note that, with suitable choices, we can in fact arrange the resulting quantum permutation to be irreducible.

\section{Infinite quantum permutation groups} \label{secinfiniteqpg}

In this section we shall restrict our attention to the class of finite dimensional quantum permutations of a set and the associated quantum groups. 

If $ \sigma = (\HH_\sigma, u^\sigma) $ is a finite dimensional quantum permutation of a set $ X $ then it is a dualisable object in the $ C^* $-tensor category of all quantum 
permutations of $ X $, and its dual is given by the contragredient quantum permutation $ \overline{\sigma} $. That is, 
the full subcategory of all finite dimensional quantum permutations of $ X $ forms naturally a rigid $ C^* $-tensor category, compare \cite{NTlecturenotes}. 
This category admits a tautological fiber functor to the category of finite dimensional Hilbert spaces. 

\begin{definition} \label{defsymplus}
Let $ X $ be a set. The quantum permutation group $ \Sym^+(X) $ is the discrete quantum group obtained from the rigid $ C^* $-tensor category $ \SSym^+(X) $ 
of finite dimensional quantum permutations of $ X $ together with its tautological fiber functor via Tannaka-Krein reconstruction. 
\end{definition} 

By construction, the underlying $ C^* $-algebra of functions on $ \Sym^+(X) $ can be written as the $ C^* $-direct sum of matrix algebras
$$
C_0(\Sym^+(X)) = \bigoplus_{\sigma} B(\HH_\sigma), 
$$
where the direct sum is taken over the set of isomorphism classes of irreducible objects in $ \SSym^+(X) $. According to Lemma \ref{onedim} one obtains a canonical 
surjective $ * $-homomorphism $ C_0(\Sym^+(X)) \rightarrow C_0(\Sym(X)) $ by projecting onto the direct sum of 
all one-dimensional matrix blocks. This shows already that the $ C^* $-algebra $ C_0(\Sym^+(X)) $ is not separable if $ X $ is infinite. 

To describe the quantum group structure of $ C_0(\Sym^+(X)) $ it is convenient to work with the \emph{universal quantum permutation} over $ X $, by which 
we mean the family $ u = (u_{xy})_{x,y \in X} $ of elements $ u_{xy} \in M(C_0(\Sym^+(X))) $ with components $ u^\sigma_{xy} $. The coproduct of $ C_0(\Sym^+(X)) $ is 
the uniquely determined nondegenerate $ * $-homomorphism $ \Delta: C_0(\Sym^+(X)) \rightarrow M(C_0(\Sym^+(X)) \otimes C_0(\Sym^+(X))) $ satisfying 
$$
\Delta(u_{xy}) = \sum_{z \in X} u_{xz} \otimes u_{zy} 
$$
for all $ x,y \in X $, with the sum on the right hand side converging in the strict topology. 
The counit $ \epsilon: C_0(\Sym^+(X)) \rightarrow \mathbb{C} $ is the unique $ * $-homomorphism satisfying $ \epsilon(u_{xy}) = \delta_{xy} $. 
Moreover the formula $ S(u_{xy}) = u_{yx} $ determines a $ * $-antihomomorphism $ C_0(\Sym^+(X)) \rightarrow C_0(\Sym^+(X)) $ encoding the antipode. In particular, 
the quantum permutation group $ \Sym^+(X) $ is unimodular.  

As already indicated above, the quantum group $ \Sym^+(X) $ is ``large'' if $ X $ is infinite. In fact, it is easy to check that there are 
uncountably many mutually non-isomorphic irreducible quantum permutations in every positive dimension in this case.  

Let us also note that this quantum group does not have any of the standard approximation properties. More precisely, if $ X $ is infinite then $ \Sym^+(X) $ 
does not have the Haagerup property and is not weakly amenable. This follows simply from the corresponding facts for the classical group $ \Sym(X) $ together 
with the observation that these properties pass to quantum subgroups. 
Indeed, the passage of the Haagerup property to quantum subgroups of discrete quantum groups is a special case of Proposition 6.8 in \cite{DFSWhaagerup}, 
noting that the argument given there does not rely on second countability assumptions. 
Inheritance of weak amenability by quantum subgroups of discrete quantum groups is shown in Proposition 3.1 of \cite{Freslonpermanence}.
At the same time we remark that $ \Sym^+(X) $ does not have property (T) since it is not finitely generated, compare \cite{FIMA_propertyt}. 

We shall now explain how infinite quantum permutations can be viewed as \emph{universal quantum symmetries}, in analogy 
to the situation for quantum automorphism groups of finite dimensional $ C^* $-algebras. However, in contrast to the compact quantum groups considered in \cite{Wangqsymmetry}, 
we obtain universal objects among discrete quantum groups in our setting. 

Let us fix some notation and terminology. Assume that $ B $ is a $ C^* $-algebra with positive part $ B_+ $, and let $ \theta: B_+ \rightarrow [0,\infty] $ be a 
faithful proper weight \cite{KUSTERMANS_kmsweights}. 
We write $ \M_\theta^+ = \{b \in B_+ \mid \theta(b) < \infty \} $ for the set of $ \theta $-integrable elements of $ B $. 
If $ \beta: B \rightarrow M(C_0^\red(G) \otimes B) $ is an action of a locally compact quantum group $ G $ on $ B $ then $ \theta $ is called invariant 
with respect to $ \beta $ if $ \theta((\omega \otimes \id)\beta(b)) = \theta(b) \omega(1) $ for all $ b \in \M_\theta^+ $ and all positive linear
functionals $ \omega \in C_0^\red(G)_+^* $. 

We shall say that a discrete quantum group $ G $ is the \emph{discrete quantum symmetry group} of $ (B, \theta) $ if there exists an 
action $ \beta: B \rightarrow M(C_0(G) \otimes B) $ such that $ \theta $ is invariant with respect to $ \beta $ and the following universal property is satisfied. 
If $ H $ is an arbitrary discrete quantum group together with an action $ \gamma: B \rightarrow M(C_0(H) \otimes B) $ 
such that $ \theta $ is invariant with respect to $ \gamma $, then there exists a unique morphism of quantum groups $ \iota: H \rightarrow G $ such that the diagram 
$$
\xymatrix{
B \ar@{->}[r]^{\!\!\!\!\!\!\!\!\!\!\!\!\!\!\!\!\!\!\!\!\!\! \beta} \ar@{->}[rd]_{\gamma} & M(C_0(G) \otimes B) \ar@{->}[d]^{\iota^* \otimes \id} \\ 
& M(C_0(H) \otimes B) 
}
$$
is commutative. Here $ \iota^*: C_0(G) \rightarrow M(C_0(H)) $ denotes the nondegenerate $ * $-homomorphism corresponding to the morphism $ \iota $. 

It is straightforward to check that the discrete quantum symmetry group of a $ C^* $-algebra $ B $ is uniquely determined up to isomorphism. 
If $ X $ is a set and $ B = C_0(X) $ is equipped with the weight induced by counting measure we shall call the corresponding discrete quantum symmetry group 
the \emph{universal discrete quantum group acting on $ X $}.
The following result shows in particular that this quantum group indeed exists. 

\begin{prop} \label{quantumsymmetryuniversal}
Let $ X $ be a set. The quantum permutation group $ \Sym^+(X) $ is the universal discrete quantum group acting on $ X $. 
\end{prop} 

\begin{proof} 
If we write $ e_x \in C_0(X) $ for the characteristic function based at $ x \in X $ then the formula $ \beta(e_x) = \sum_{y \in X} u_{xy} \otimes e_y $ 
determines a nondegenerate $ * $-homomorphism  $ \beta: C_0(X) \rightarrow M(C_0(\Sym^+(X)) \otimes C_0(X)) $, 
and it is straightforward to check that this yields an action of $ \Sym^+(X) $ on $ B = C_0(X) $. 

Now assume that $ \gamma: B \rightarrow M(C_0(H) \otimes B) $ is an action of a discrete quantum group $ H $ on $ B $. Then we can 
write $ \gamma(e_x) = \sum_{y \in X} v_{xy} \otimes e_y $ for uniquely determined elements $ v_{xy} \in M(C_0(H)) $. We claim that the elements $ v_{xy} $, 
or rather their components in each matrix block inside $ C_0(H) $, satisfy the defining properties for a quantum permutation. First note that 
from $ e_x^2 = e_x = e_x^* $ we get $ v_{xy}^2 = v_{xy} = v_{xy}^* $ for all $ x,y \in X $. 
Since $ \gamma $ is nondegenerate we have $ \sum_{x \in X} v_{xy} = 1 $ for all $ y \in X $. 
Moreover, we obtain the relation $ \sum_{y \in X} v_{xy} = 1 $ for all $ x \in X $ from the fact that the canonical weight on $ B $ is required to be 
invariant with respect to $ \gamma $. 

Since $ C_0(H) $ is a direct sum of matrix algebras this allows us to define a nondegenerate $ * $-homomorphism $ \iota^*: C_0(\Sym^+(X)) \rightarrow M(C_0(H)) $ 
such that $ \iota^*(u_{xy}) = v_{xy} $ for all $ x, y \in X $. It is straightforward to check that $ \iota^* $ is compatible with the comultiplications 
and satisfies $ (\iota^* \otimes \id) \beta = \gamma $. Moreover $ \iota^* $ is uniquely determined by these properties. 
\end{proof}

\section{Finitary quantum permutations} \label{seclocallyfinite}

In this section we study the quantum subgroup of the quantum permutation group $ \Sym^+(X) $ given by all finitary quantum permutations of the set $ X $. 

Classically, one obtains a subgroup $ \Sigma(X) \subset \Sym(X) $ by considering all finitary permutations, that is, permutations which move only finitely many 
points of $ X $. Equivalently, one can view $ \Sigma(X) = \varinjlim_{F \subset X} \Sigma(F) $ as the direct limit of the permutation groups $ \Sym(F) = \Sigma(F) $ 
taken over the finite subsets $ F \subset X $. 

This translates easily to $ \Sym^+(X) $. More precisely, consider the full subcategory of the $ C^* $-tensor category $ \SSym^+(X) $ formed by all quantum 
permutations $ \sigma = (\HH_\sigma, u^\sigma) $ for which there exists a finite set $ F \subset X $ such that $ u^\sigma_{xy} \neq \delta_{xy} $ only for $ x,y \in F $. 
In this case we say that $ \sigma $ \emph{moves only finitely many points}, or is \emph{finitary}. It is straightforward to check that the collection of all finitary quantum 
permutations forms a $ C^* $-tensor subcategory of $ \SSym^+(X) $. 

\begin{definition} 
Let $ X $ be a set. The finitary quantum permutation group of $ X $ is the discrete quantum group $ \Sigma^+(X) $ obtained from the rigid $ C^* $-tensor category 
of all finite dimensional finitary quantum permutations of $ X $ together with its tautological fiber functor via Tannaka-Krein reconstruction. 
\end{definition} 
 
We clearly have $ \Sym^+(X) = \Sigma^+(X) $ iff $ X $ is finite. 
In the same way as in the classical case one can write $ \Sigma^+(X) = \varinjlim_{F \subset X} \Sigma^+(F) $ as direct limit of the quantum permutation groups of the finite 
subsets $ F \subset X $ in general. For many purposes, this reduces the study of finitary quantum permutations to the case of finite sets.

Let us therefore study the quantum permutation group $ \Sym^+(X) = \Sigma^+(X) $ of a finite set $ X $ in more detail. We consider $ {\bf n} = \{1, \dots, n \} $ 
for $ n \in \mathbb{N} $ and write $ \Sigma_n^+ $ instead of $ \Sigma^+({\bf n}) $ for the associated quantum group. Note that we obtain a canonical 
isomorphism $ S_n^+ \cong \Sigma_n^+ $ for $ n = 1,2,3 $ since there are no non-classical quantum permutations for these values of $ n $. 

For $ n > 3 $ the quantum groups $ S_n^+ $ and $ \Sigma_n^+ $ are no longer isomorphic. 
By construction, the $ C^* $-tensor category of unitary corepresentations of $ \Sigma_n^+ $ identifies with the $ C^* $-tensor category $ \Rep(C^\full(S_n^+)) $ of finite 
dimensional unital $ * $-representations of $ C^\full(S_n^+) $. We may interpret this as saying that $ \Sigma_n^+ $ is the \emph{discretisation} of $ S_n^+ $ in the following sense. 

\begin{definition} \label{defdiscretisation}
Let $ G $ be a compact quantum group. The discretisation of $ G $ is the discrete quantum group $ G_\delta $ associated to the rigid $ C^* $-tensor 
category $ \Rep(C^\full(G)) $ of finite dimensional unital $ * $-representations of $ C^\full(G) $ via Tannaka-Krein reconstruction. 
\end{definition} 

Here the tensor structure on $ \Rep(C^\full(G)) $ is given by $ \pi \tp \eta = (\pi \otimes \eta) \Delta $, 
where $ \Delta: C^\full(G) \rightarrow C^\full(G) \otimes C^\full(G) $ is the comultiplication. The contragredient $ \overline{\pi} $ of a 
representation $ \pi: C^\full(G) \rightarrow B(\HH) $ in $ \Rep(C^\full(G)) $ is the representation on the conjugate Hilbert space $ \overline{\HH} $ 
given by $ \overline{\pi}(f)(\overline{\xi}) = \overline{\pi(R(f^*))(\xi)} $ where $ R $ is the unitary antipode of $ C^\full(G) $. 
We note that finite dimensional $ * $-representations of $ C^\full(G) $ factor through the Kac quotient, see \cite{SOLTAN_bohr}. 

The dual of the discretisation $ G_\delta $ of a compact quantum group $ G $ in the sense of Definition \ref{defdiscretisation} identifies with  
the Bohr compactification \cite{SOLTAN_bohr} of the dual discrete quantum group $ \hat{G} $. 
That is, discretisation in the sense of Definition \ref{defdiscretisation} is in fact nothing new, and the only difference to \cite{SOLTAN_compactifications}, \cite{SOLTAN_bohr} 
is that we focus our attention on the dual side. We note that the discretisation procedure is functorial, that is, if $ \iota: H \rightarrow G $ 
is a morphism of compact quantum groups then there is an induced morphism $ \iota_\delta: H_\delta \rightarrow G_\delta $ 
between the discretisations. 

There is a canonical morphism of locally compact quantum groups $ G_\delta \rightarrow G $, implemented by the 
unital $ * $-homomorphism $ C^\full(G) \rightarrow M(C_0(G_\delta)) $ 
obtained by considering the direct sum of all irreducible finite dimensional $ * $-representations of $ C^\full(G) $ up to equivalence. If $ G $ is a classical compact group 
then the discretisation of $ G $ in the sense of Definition \ref{defdiscretisation} is given by the group $ G $ with the discrete topology, compare \cite{SOLTAN_bohr}. 

Let $ X = {\bf n} $ for some $ n \in \mathbb{N} $ as above, and assume that $ \gamma: C(X) \rightarrow C(X) \otimes C^\red(G) $ is an action of a compact quantum group $ G $ on $ X $. 
Note that the image of $ \gamma $ is automatically contained in $ C(X) \otimes \Poly(G) $, so that it does not really matter whether we consider full or reduced function algebras here. 
We obtain an induced action $ \gamma_\delta: C(X) \rightarrow M(C(X) \otimes C_0(G_\delta)) $ via the canonical morphism $ G_\delta \rightarrow G $. 
It is easy to check that the corresponding morphism $ G_\delta \rightarrow \Sym^+(X) $ in Proposition \ref{quantumsymmetryuniversal} is nothing but the discretisation of the 
morphism $ G \rightarrow S_n^+ $ obtained from the universal property of $ S_n^+ $. 

In general, passage to the discretisation leads to a significant loss of information. However, this is not quite the case for 
quantum permutation groups. In fact, according to work of Brannan-Chirvasitu-Freslon \cite{BRANNAN_CHIRVASITU_FRESLON_matrixmodels} the algebra $ \Poly(S_n^+) $ of polynomial 
functions on $ S_n^+ $ is residually finite dimensional. 
One can rephrase this as saying that the canonical $ * $-homomorphism $ \Poly(S_n^+) \rightarrow M(C_0(\Sigma_n^+)) $ is injective. 

Let us exhibit some basic properties of the discrete quantum groups $ \Sigma_n^+ $, starting with the simplest nontrivial 
case $ n = 4 $. Recall from \cite{BBpauli}, \cite{BBfourpoints} that $ S_4^+ \cong SO_{-1}(3) $, where $ C(SO_{-1}(3)) $ is the universal $ C^* $-algebra 
generated by elements $ v_{ij} $ for $ 1 \leq i,j \leq 3 $ such that $ (v_{ij}) $ is an orthogonal matrix satisfying the relations 
\begin{align*}
v_{ij} v_{ik} = -v_{ik} v_{ij} &\text{ and } v_{ji} v_{ki} = -v_{ki} v_{ji} \text{ if } j \neq k \\
v_{ij} v_{kl} = v_{kl} v_{ij} &\text{ if } i\neq k \text{ and } j \neq l \\
\sum_{\sigma \in S_3} v_{1 \sigma(1)} &v_{2 \sigma(2)} v_{3 \sigma(3)} = 1. 
\end{align*}
The comultiplication of $ C(SO_{-1}(3)) $ is given by $ \Delta(v_{ij}) = \sum_k v_{ik} \otimes v_{kj} $. 

There is an injective $ * $-homomorphism $ \theta: C(SO_{-1}(3)) \rightarrow M_4(C(SO(3)) $ given by 
$$ 
\theta(v_{ij}) = t_i \otimes t_j \otimes x_{ij} 
$$ 
on the generators, where $ x_{ij} \in C(SO(3)) $ are the coordinate functions and 
$$
t_1 = 
\begin{pmatrix} 
i & 0 \\
0 & - i
\end{pmatrix},
\qquad 
t_2 = 
\begin{pmatrix} 
0 & 1 \\
-1 & 0
\end{pmatrix}, 
\qquad 
t_3 = 
\begin{pmatrix} 
0 & -i \\
-i & 0
\end{pmatrix}
$$
in $ M_2(\mathbb{C}) $ are the Pauli matrices. 

We shall show that $ \Sigma_4^+ $ is non-amenable, essentially by using the above relationship between $ SO_{-1}(3) $ and $ SO(3) $, and 
the fact that the latter contains free subgroups. However, in order to make this precise we need some preparations. 

For $ g \in SO(3) $ let us write $ \pi_g: C(S_4^+) \rightarrow M_4(\mathbb{C}) $ for the $ * $-representation obtained by composing the 
isomorphism $ C(S_4^+) \cong C(SO_{-1}(3)) $ from \cite{BBfourpoints} with the map $ \theta $, followed by evaluation at $ g $. 
By slight abuse of notation we will also view $ \pi_g $ as a representation of $ C(SO_{-1}(3)) $ in the sequel. 
For every irreducible $ * $-representation $ \rho $ of $ C(S_4^+) $ there exists an element $ g \in SO(3) $ such that $ \rho $ 
is the restriction of the evaluation representation $ \pi_g $ to an invariant subspace of $ \mathbb{C}^4 $. 
We shall say that $ \rho $ is a \emph{representation over $ g $} in this case. 

For instance, the evaluation representation $ \pi_e $ at the identity element $ e $ decomposes 
as a direct sum $ \pi_e \cong \chi_0 \oplus \chi_1 \oplus \chi_2 \oplus \chi_3 $ of four distinct characters, such that $ \chi_0 = \epsilon $ is the counit 
and $ \chi_k $ for $ 1 \leq k \leq 3 $ is determined by sending $ v = (v_{ij}) $ to the element $ d_k \in SO(3) $ given by reflection about the $ k $-th coordinate axis. 
We may thus identify these characters with the elements of the Klein subgroup $ D \subset SO(3) $ of diagonal matrices in a natural way, such that the counit corresponds 
to the identity matrix $ d_0 \in D $. 

Consider the action of $ D \times D $ on $ SO(3) $ given by left and right multiplication, so that $ (c,d) \cdot g = c g d^{-1} $ for $ c,d \in D $ and $ g \in SO(3) $, 
and let us write $ [g] \subset SO(3) $ for the orbit of $ g \in SO(3) $ under this action. 

\begin{lemma} \label{isomorphismlemma} 
Let $ g, h \in SO(3) $. If $ [g] = [h] $ then the representations $ \pi_g $ and $ \pi_h $ are equivalent. 
\end{lemma} 

\begin{proof} 
Let us first assume that we can write $ g = d h $ for some $ d \in D $. If $ d = e $ there is nothing to prove, so it suffices to consider the case that $ d = d_i $ for 
some $ 1 \leq i \leq 3 $, where we recall that $ d_i $ is reflection about the $ i $-th coordinate axis. It is straightforward to check that the operator $ t_i \otimes 1 $ 
defines an intertwiner between $ \pi_g $ and $ \pi_h $ in this case. 
In a similar way one can treat the case $ g = h d $ for $ d \in D $, and combining these two cases yields the claim. 
\end{proof}  

Let us say that the \emph{packet} associated to $ g \in SO(3) $ is the set $ \Pi_{[g]} \subset \Rep(C(S_4^+)) $ of all irreducible representations which can be realised as 
representations over elements of $ [g] $. 

\begin{lemma} \label{packetlemma} 
Every irreducible representation of $ C(S_4^+) $ is contained in a unique packet. Each packet $ \Pi_{[g]} $ contains at most $ 4 $ irreducible representations, 
and these can all be realised over any point of $ [g] \subset SO(3) $. 
\end{lemma} 

\begin{proof} 
Assume that $ \rho \in \Rep(C(S_4^+)) $ can be written as a subrepresentation of $ \pi_g $ and $ \pi_h $ for $ g,h \in SO(3) $. 
Since the matrices $ t_i \otimes t_j $ square to the identity we obtain 
$$ 
h_{ij}^2 = \|\pi_h(v_{ij}^2)\| = \|\rho(v_{ij}^2)\| = \|\pi_g(v_{ij}^2)\| = g_{ij}^2 
$$
for all $ 1 \leq i,j \leq 3 $, or equivalently, $ |g_{ij}| = |h_{ij}| $. 
A direct inspection shows that one can then write $ g = c h d $ for suitable diagonal matrices $ c, d $ with entries in $ \pm 1 $, 
and one can in fact take both $ c, d $ from $ D $. Hence we obtain $ [g] = [h] $. 

The second part of the claim follows from Lemma \ref{isomorphismlemma} and the fact that each $ \pi_g $ contains at most $ 4 $ distinct irreducible representations. 
\end{proof} 

Next we discuss the behaviour of the map $ \theta: C(SO_{-1}(3)) \rightarrow M_4(C(SO(3)) $ with respect to the comultiplications of $ C(SO_{-1}(3)) $ and $ C(SO(3)) $. 

\begin{lemma} \label{fusionlemma} 
We have 
$$ 
\pi_g \tp \pi_h \cong \bigoplus_{d \in D} \pi_{gdh} 
$$ 
for all $ g, h \in SO(3) $. 
\end{lemma}  

\begin{proof} 
By definition, the representation $ \pi_g \tp \pi_h = (\pi_g \otimes \pi_h) \Delta $ maps the generator $ v_{ij} $ of $ C(SO_{-1}(3)) $ to 
$$ 
\sum_k \pi_g(v_{ik}) \otimes \pi_h(v_{kj}) = \sum_k g_{ik} h_{kj} \,t_i \otimes t_k \otimes t_k \otimes t_j. 
$$
The middle two tensor factors in this expression can be decomposed in the same way as for the representation $ \pi_e $. 
More precisely, let us write $ p_i \in M_4(\mathbb{C}) $ for $ 0 \leq i \leq 3 $ for the projections onto the 
irreducible components of the representation $ \pi_e $, so that the restriction of $ \pi_e $ to $ p_i(\mathbb{C}^4) \subset \mathbb{C}^4 $ is 
given by the character $ \chi_i $ associated with $ d_i \in D $. 
Then the operators $ 1 \otimes p_i \otimes 1 $ are intertwiners of $ \pi_g \tp \pi_h $, and they implement a direct sum decomposition 
of $ \pi_g \tp \pi_h $ into the representations $ \pi_{gd_ih} $ as claimed. 
\end{proof} 

Finally, we recall the F\o lner type characterisation of amenability for discrete quantum groups from \cite{KYED_betticoamenable}. 
A countable discrete quantum group $ G $ is amenable iff for every nonempty finite subset $ S \subset \Irr(\hat{G}) $ and $ \epsilon > 0 $ there exists a finite 
set $ F \subset \Irr(\hat{G}) $ such that 
$$
\sum_{t \in \partial_S(F)} \dim(t)^2 < \epsilon \sum_{t \in F} \dim(t)^2,   
$$ 
where the boundary $ \partial_S(F) $ of $ F $ relative to $ S $ is defined by 
\begin{align*}
\partial_S(F) = \{t \in F &\mid \exists r \in S \text{ such that } \supp(tr) \not\subset F \} \\
&\cup \{t \in \Irr(\hat{G}) \setminus F \mid \exists r \in S \text{ such that } \supp(tr) \not\subset \Irr(\hat{G}) \setminus F \}.  
\end{align*} 
Here the support $ \supp(f) $ of an element $ f = \sum_{r \in \Irr(\hat{G})} \lambda_r r \in \mathbb{Z}[\Irr(\hat{G})] $ is the set of 
all $ r \in \Irr(\hat{G}) $ such that $ \lambda_r \neq 0 $. 

We are now ready to prove the following statement. 

\begin{theorem}
The discrete quantum group $ \Sigma_4^+ $ is non-amenable. 
\end{theorem} 

\begin{proof} 
Let us choose a copy $ H $ of the free group on two generators inside the discretisation $ SO(3)_\delta $ of $ SO(3) $. 
Moreover let $ K \subset SO(3)_\delta $ be the subgroup generated by $ D $ and $ H $. 

Let $ U \subset \Sigma_4^+ $ be the quantum subgroup generated by all irreducible representations of $ C(S_4^+) $ over elements of $ K $. By construction, $ U $ is 
a countable discrete quantum group, and using Lemma \ref{fusionlemma} we see that $ \Irr(\hat{U}) $ consists precisely of the 
corepresentations of $ \Sigma_4^+ $ corresponding to the irreducible representations of $ C(S_4^+) $ which are contained in $ \pi_k $ for some $ k \in K $. 

For $ S \subset K $ we let 
$$ 
S^+ = \bigcup_{s \in S} \Pi_{[s]} \subset \Irr(\hat{U}), 
$$ 
and for $ T \subset \Irr(\hat{U}) $ we write 
$$ 
T^- = \{k \in K \mid \exists \eta \in T \text{ such that } \eta \in \Pi_{[k]} \} \subset K. 
$$ 
Note that if $ S $ is finite then the same is true for $ S^+ $ by Lemma \ref{packetlemma}, and similarly $ T^- $ is finite provided $ T $ is. 

Consider a nonempty finite set $ S \subset K $, a finite set $ F \subset \Irr(\hat{U}) $, and let $ x \in F^- $. Then we find an irreducible corepresentation $ \pi $ over $ x $ 
which is contained in $ F $. 
Assume that $ xs \notin F^- $ for some $ s \in S $, and let $ \sigma \in \Irr(\hat{U}) $ be over $ s $. Then $ \sigma \in S^+ $ by the definition of $ S^+ $. 
If the support of $ \pi \tp \sigma $ is entirely contained in $ F $ for all such $ \sigma $  then we get $ xs \in F^- $. This is impossible, and hence $ \pi \in \partial_{S^+}(F) $. 
Similarly, if $ x \in K \setminus F^- $ then every irreducible corepresentation $ \pi $ over $ x $ is contained in $ \Irr(\hat{U}) \setminus F $. 
Assume that $ xs \notin K \setminus F^- $ for some $ s \in S $, or equivalently $ xs \in F^- $. 
If the support of $ \pi \tp \sigma $ is entirely contained in $ \Irr(\hat{U}) \setminus F $ for all such $ \pi $ and all $ \sigma \in S^+ $ over $ s $ then we 
get $ xs \in K \setminus F^- $. This is impossible, and therefore $ \pi \in \partial_{S^+}(F) $ for some $ \pi $. 
Combining these considerations we obtain
$$
|\partial_S(F^-)| \leq 16 |\partial_{S^+}(F)|, 
$$
using that every class $ [x] \subset K $ contains at most $ 16 $ elements. 

Now assume that $ U $ is amenable. Then for every nonempty finite set $ S \subset K $ and every $ \epsilon > 0 $ we 
find a finite set $ F \subset \Irr(\hat{U}) $ such that
$$
|\partial_{S^+}(F)| \leq \sum_{t \in \partial_{S^+}(F)} \dim(t)^2 < \epsilon \sum_{t \in F} \dim(t)^2 \leq 16 \epsilon |F|, 
$$
noting that $ 1 \leq \dim(t) \leq 4 $ for every $ t \in \Irr(\hat{U}) $. Since each packet contains at most $ 4 $ irreducibles we also have $ |F| \leq 4 |F^-| $, 
and hence
$$ 
|\partial_S(F^-)| \leq 16 |\partial_{S^+}(F)| < 256 \epsilon |F| \leq 1024 \epsilon |F^-|. 
$$
This contradicts the fact that $ K $ is not amenable, and finishes the proof. 
\end{proof} 

Since $ \Sigma_4^+ $ fails to be amenable the same holds true for $ \Sigma_n^+ $ for any $ n \geq 4 $, and also for the finitary quantum 
permutation group $ \Sigma^+(X) = \varinjlim_{F \subset X} \Sigma^+(F) $ associated with an infinite set $ X $. 
Of course, this is in contrast to the situation for the corresponding classical groups. 

Note that none of the quantum groups $ \Sigma_n^+ $ for $ n \geq 4 $ and $ \Sigma^+(X) $ for an infinite set $ X $ have property (T) because they are not finitely generated. 
In fact, our above analysis shows in particular that $ \Sigma_4^+ $ is uncountable.

\section{Free wreath products} \label{secfreewreath}

Utilising infinite quantum permutation groups in the sense of Definition \ref{defsymplus}, we shall now explain a variant of the free wreath product construction 
introduced by Bichon, compare \cite{Bichonfreewreathproducts}. 

Throughout we fix a discrete quantum group $ G $ and a set $ X $. By an \emph{$ X $-free wreath corepresentation of $ G $} we mean a pair $ \pi = (\pi^G, \pi^X) $ consisting of a
family $ \pi^G = (\pi^G_x)_{x \in X} $ of nondegenerate $ * $-representations of $ C_0(G) $ on the same Hilbert space $ \HH_\pi $ together with a quantum 
permutation $ \pi^X = (\pi^X_{xy}) $ of $ X $ on $ \HH_\pi $, such that 
$$
\pi^G_x(f) \pi^X_{xy} = \pi^X_{xy} \pi^G_x(f) 
$$ 
for all $ f \in C_0(G) $ and $ x,y \in X $. 

If $ \pi = (\pi^G, \pi^X), \rho = (\rho^G, \rho^X) $ are $ X $-free wreath corepresentations of $ G $ then their tensor product is defined by 
the family of $ * $-representations $ (\pi \tp \rho)^G_x $ on $ \HH_\pi \otimes \HH_\rho $ given by 
\begin{align*}
(\pi \tp \rho)^G_x(f) = \sum_{y \in X} (\pi^X_{xy} \otimes 1)(\pi^G_x \otimes \rho^G_y)\Delta(f) = \sum_{y \in X} (\pi^G_x \otimes \rho^G_y) \Delta(f) (\pi^X_{xy} \otimes 1)
\end{align*}
and the quantum permutation $ \pi^X \tp \rho^X $. Since 
\begin{align*}
(\pi \tp \rho)^G_x(f) (\pi \tp \rho)^X_{xy}
&= \sum_{v,w \in X} (\pi^X_{xv} \otimes 1)(\pi^G_x \otimes \rho^G_v)\Delta(f) (\pi^X_{xw} \otimes \rho^X_{wy}) \\
&= \sum_{v,w \in X} (\pi^X_{xv} \pi^X_{xw} \otimes 1)(\pi^G_x \otimes \rho^G_v)\Delta(f) (1 \otimes \rho^X_{wy}) \\
&= \sum_{w \in X} (\pi^X_{xw} \otimes \rho^X_{wy})(\pi^G_x \otimes \rho^G_w)\Delta(f) \\
&= \sum_{v,w \in X} (\pi^X_{xw} \otimes \rho^X_{wy})(\pi^G_x \otimes \rho^G_v)\Delta(f) (\pi^X_{xv} \otimes 1) \\
&= (\pi \tp \rho)^X_{xy} (\pi \tp \rho)^G_x(f) 
\end{align*}
we see that $ \pi \tp \rho $ is again an $ X $-free wreath corepresentation of $ G $. 
The trivial $ X $-free wreath corepresentation $ \epsilon $ of $ G $ on the Hilbert space $ \mathbb{C} $ is given by the counit of $ C_0(G) $ for all $ x \in X $ 
together with the trivial quantum permutation. 

An intertwiner of $ X $-free wreath corepresentations $ \pi, \rho $ of $ G $ is a bounded linear operator $ T: \HH_\pi \rightarrow \HH_\rho $ such that $ T \pi^G_x = \rho^G_x T $ 
and $ T \pi^X_{xy} = \rho^X_{xy} T $ for all $ x, y \in X $. It is straightforward to check that the collection of all $ X $-free wreath corepresentations of $ G $ forms 
a $ C^* $-tensor category. 

By definition, the contragredient of an $ X $-free wreath corepresentation $ \pi = (\pi^G, \pi^X) $ of $ G $ acts on the conjugate Hilbert space of $ \HH_\pi $ and is determined  
by $ \overline{\pi} = (\overline{\pi}^G, \overline{\pi}^X) $ where 
$$
\overline{\pi}^G_x(f)(\overline{\xi}) = \sum_{y \in X} \overline{\pi^G_y(R(f^*)) \pi^X_{yx}(\xi)}
$$
and $ \overline{\pi}^X_{xy}(\overline{\xi}) = \overline{\pi^X_{yx}(\xi)} $. Here $ R $ denotes the unitary antipode of $ C_0(G) $. 
One checks that this is indeed a well-defined $ X $-free wreath corepresentation of $ G $. 

We say that an $ X $-free wreath corepresentation of $ G $ is finite dimensional if its underlying Hilbert space is, and we write $ \Corep(G) \Wr_* \SSym^+(X) $ 
for the corresponding full $ C^* $-tensor subcategory of the category of $ X $-free wreath corepresentations of $ G $. 

\begin{lemma} \label{dualisable}
Let $ G $ be an unimodular discrete quantum group and let $ X $ be a set. Then every object $ \pi $ of the $ C^* $-tensor category $ \Corep(G) \Wr_* \SSym^+(X) $ 
is dualisable, and the dual of $ \pi $ is given by the contragredient $ \overline{\pi} $. 
\end{lemma} 

\begin{proof} 
Fix a finite dimensional $ X $-free wreath corepresentation $ \pi = (\pi^G, \pi^X) $ of $ G $ and choose an orthonormal basis $ e_1, \dots, e_n $ of the underlying 
Hilbert space $ \HH_\pi $. If $ \overline{e}_1, \dots, \overline{e}_n $ denotes the dual basis of $ \HH_{\overline{\pi}} \cong \HH_\pi^* $ we obtain 
bounded linear operators $ ev_\pi: \HH_{\overline{\pi}} \otimes \HH_\pi \rightarrow \mathbb{C} $ and $ db_\pi: \mathbb{C} \rightarrow \HH_\pi \otimes \HH_{\overline{\pi}} $ by 
$$
ev_\pi(\overline{e}_i \otimes e_j) = \delta_{ij} = \bra e_i, e_j \ket, \qquad db_\pi(1) = \sum_{i = 1}^n e_i \otimes \overline{e}_i.  
$$ 
We shall work with the multiplier Hopf algebra $ C_c(G) $ of finitely supported elements in $ C_0(G) $, and use Sweedler notation $ \Delta(f) = f_{(1)} \otimes f_{(2)} $ 
for the coproduct of $ f \in C_c(G) $. Keeping in mind that any $ \eta \in \HH_\pi $ can be written in the form $ \eta = \pi^G_v(g_y)(\eta_y) $ for suitable 
elements $ \eta_y \in \HH_\pi $ and $ g_y \in C_c(G) $ for $ y \in X $, we compute 
\begin{align*}
ev_\pi (\overline{\pi} \tp \pi)^G_x(f)(\overline{\xi} \otimes \eta)
&= \sum_{y \in X} ev_\pi ((\overline{\pi}^X_{xy} \otimes 1)(\overline{\pi}^G_x \otimes \pi^G_y) \Delta(f)(\overline{\xi} \otimes \eta)) \\
&= \sum_{y,v \in X} ev_\pi ((\overline{\pi}^X_{xy} \overline{\pi}^X_{xv} \otimes 1)(\overline{\pi^G_v(S(f_{(1)})^*)(\xi)} \otimes \pi^G_y(f_{(2)})(\eta)) \\
&= \sum_{y \in X} ev_\pi (\overline{\pi_{yx} \pi^G_y(S(f_{(1)})^*)(\xi)} \otimes \pi^G_y(f_{(2)})(\eta)) \\
&= \sum_{y \in X} ev_\pi (\overline{\pi^G_y(S(f_{(1)})^*)(\pi_{yx} \xi)} \otimes \pi^G_y(f_{(2)})(\eta)) \\
&= \sum_{y \in X} \bra \pi^G_y(S(f_{(1)})^*)(\pi_{yx} \xi), \pi^G_y(f_{(2)})(\eta) \ket \\
&= \sum_{y \in X} \bra \pi_{yx} \xi, \pi^G_y(S(f_{(1)}) f_{(2)})(\eta) \ket \\
&= \sum_{y \in X} \epsilon(f) \bra \pi^X_{yx} \xi, \eta \ket \\
&= \epsilon^G_x(f) ev_\pi(\overline{\xi} \otimes \eta) 
\end{align*}
for $ \xi, \eta \in \HH_\pi $. Here we also use that the unitary antipode $ R $ agrees with the ordinary antipode $ S $ since $ G $ is unimodular. 
Similarly, using the canonical identification $ \HH_\pi \otimes \HH_{\overline{\pi}} \cong B(\HH_\pi) $ we obtain 
\begin{align*}
(\pi \tp \overline{\pi})^G_x(f) db_\pi(1)(\xi)
&= \sum_{i = 1}^n \sum_{y \in X} (\pi^X_{xy} \otimes 1)(\pi^G_x \otimes \overline{\pi}^G_y) \Delta(f)(e_i \otimes \overline{e}_i)(\xi) \\
&= \sum_{i = 1}^n \sum_{y,v \in X} (\pi^X_{xy} \pi^G_x(f_{(1)})(e_i) \otimes \overline{\pi^X_{vy} \pi^G_v(S(f_{(2)}^*))(e_i)})(\xi) \\
&= \sum_{i = 1}^n \sum_{y,v \in X} \pi^G_x(f_{(1)})(\pi^X_{xy} e_i) \bra e_i, \pi^X_{vy} \pi^G_v(S(f_{(2)}))(\xi) \ket \\
&= \sum_{i = 1}^n \sum_{y \in X} \pi^G_x(f_{(1)})(\pi^X_{xy} e_i) \bra e_i, \pi^G_x(S(f_{(2)}))(\xi) \ket \\
&= \sum_{i = 1}^n \pi^G_x(f_{(1)})(e_i) \bra e_i, \pi^G_x(S(f_{(2)}))(\xi) \ket \\
&= \sum_{i = 1}^n \epsilon(f) e_i \bra e_i, \xi \ket \\
&= db_\pi \epsilon^G_x(f)(1) (\xi)
\end{align*}
for all $ f \in C_c(G) $ and $ \xi \in \HH_\pi $. Since compatibility with the corresponding quantum permutations is obvious
we conclude that both $ ev_\pi $ and $ db_\pi $ are intertwiners. Moreover these maps clearly satisfy the zig zag equations $ (\id \otimes ev_\pi)(db_\pi \otimes \id) = \id $ 
and $ (ev_\pi \otimes \id)(\id \otimes db_\pi) = \id $. 

Swapping the roles of $ \pi $ and $ \overline{\pi} $ in the above argument we obtain the dual zig zag equations, and it follows that $ \pi $ is dualisable with 
dual $ \overline{\pi} $. 
\end{proof} 

We note that Lemma \ref{dualisable} fails if $ G $ is not unimodular, because the evaluation and dual basis morphisms for the representations $ \pi^G_x $ may depend on $ x $ 
in this case. Still we may give the following general definition. 

\begin{definition} \label{deffreewreath}
Let $ G $ be a discrete quantum group and let $ X $ be a set. The unrestricted free wreath product $ G \Wr_* \Sym^+(X) $ is the discrete quantum group associated to the 
rigid $ C^* $-tensor category of all dualisable $ X $-free wreath corepresentations of $ G $ together with its tautological fiber functor. 
\end{definition}  

Note that if $ G $ is unimodular then Lemma \ref{dualisable} shows that every finite dimensional $ X $-free wreath corepresentation of $ G $ is dualisable. 

In the same way as in the classical situation we may also define restricted free wreath products. More precisely, by a restricted $ X $-free wreath corepresentation of 
a discrete quantum group $ G $ we mean an $ X $-free wreath corepresentation $ \pi = (\pi^G, \pi^X) $ of $ G $ such that $ \pi^G_x = \epsilon $ is the counit of $ C_0(G) $ 
for all but finitely many points $ x \in X $. One checks that tensor products and contragredients of restricted $ X $-free wreath corepresentations of $ G $ 
are again restricted. 

\begin{definition} 
Let $ G $ be a discrete quantum group and let $ X $ be a set. The restricted free wreath product $ G \wwr_* \Sym^+(X) $ is the quantum subgroup of $ G \Wr_* \Sym^+(X) $ 
associated to the rigid $ C^* $-tensor category of dualisable restricted $ X $-free wreath corepresentations of $ G $ together with its tautological fiber functor. 
\end{definition}  

It follows essentially by construction that if $ G $ is a compact quantum group with discretisation $ G_\delta $ and $ X = {\bf n} = \{1,\dots, n\} $ is finite then 
there is a canonical isomorphism 
$$
(G \wr_* S_n^+)_\delta \cong G_\delta \Wr_* \Sigma_n^+ = G_\delta \wwr_* \Sigma_n^+ 
$$
of discrete quantum groups, where $ G \wr_* S_n^+ $ is the free wreath product defined by Bichon \cite{Bichonfreewreathproducts}. Here we write $ \Sigma_n^+ = \Sym^+({\bf n}) $ 
as before. 

Let us also consider an ``abelianised'' version of the free wreath product. More precisely, if $ G $ is a discrete quantum group and $ X $ a set, 
consider the full subcategory of $ X $-free wreath corepresentations $ \pi = (\pi^G, \pi^X) $ of $ G $ such that $ \pi^G_x(f) \pi_y^G(g) = \pi^G_y(g) \pi_x^G(f) $ 
for all $ x \neq y $ and $ f,g \in C_0(G) $, and with $ \pi^X $ a classical permutation of $ X $. 
We shall call such $ X $-free wreath corepresentations of $ G $ \emph{half-liberated}. 
It is straightforward to check that the class of all half-liberated $ X $-free wreath corepresentations of $ G $ is closed under taking tensor products and contragredients. 

\begin{definition} \label{deffreewreath}
Let $ G $ be a discrete quantum group and let $ X $ be a set. The unrestricted wreath product $ G \Wr \Sym(X) $ is the discrete quantum group associated to the 
rigid $ C^* $-tensor category of dualisable half-liberated $ X $-free wreath corepresentations of $ G $ and its tautological fiber functor. 
The restricted wreath product $ G \wwr \Sym(X) $ is the quantum subgroup of $ G \Wr \Sym(X) $ corresponding to the dualisable restricted half-liberated $ X $-free wreath 
corepresentations of $ G $. 
\end{definition}  

We note that the ``classical'' wreath products in Definition \ref{deffreewreath} are typically non-classical discrete quantum groups. In the case that $ X $ 
is finite, analogues of these objects in the world of compact quantum groups have recently been studied by Gromada \cite{GROMADA_quantumsymmetriescayley}.

\section{Quantum automorphisms of infinite graphs} \label{secinfinitegraphs}

In this section we extend our study of infinite quantum permutations to the case of graphs, that is, we discuss quantum symmetries of infinite graphs. 

We use the conventions and notation from section \ref{secprelim}, with the difference that we no longer require graphs to be finite. 
That is, in the sequel, by a graph $ X = (V_X, E_X) $ we mean a set $ V_X $ of vertices together with a set $ E_X \subset V_X \times V_X $ of edges such 
that $ (v,v) \notin E_X $ for all $ v \in V_X $ and $ (v,w) \in E_X $ iff $ (w,v) \in E_X $. 
We say that $ v $ and $ w $ are connected by an edge iff $ (v,w) \in E_X $, and define the degree of $ v \in V_X $ as the 
cardinality of the set of vertices which are connected to $ v $ by an edge. 
The adjacency matrix $ A_X \in M_{V_X}(\{0,1\}) $ can be viewed as a linear map $ C_c(V_X) \rightarrow C(V_X) $. 
We note that it induces a bounded linear operator $ l^2(V_X) \rightarrow l^2(V_X) $ iff $ X $ has finite degree, that is, iff the degrees of the vertices of $ X $ 
are uniformly bounded. As in section \ref{secprelim} we shall write $ \rel $ for the function on pairs of vertices which describes the adjacency relation, taking the 
values \emph{equal}, \emph{adjacent}, \emph{distinct and non-adjacent}. We denote by $ \Aut(X) $ the automorphism group of $ X $, that is, the group of all 
permutations of $ V_X $ preserving the adjacency relation. 

\begin{definition} \label{defqaut}
Let $ X = (V_X, E_X) $ be a graph. A quantum automorphism of $ X $ is a quantum permutation $ \sigma = (\HH_\sigma, u^\sigma) $ of $ V_X $ such that 
$$ 
u^\sigma_{x_1y_1} u^\sigma_{x_2y_2} = 0 
$$ 
if $ \rel(x_1,x_2) \neq \rel(y_1,y_2) $. 
\end{definition} 

It is straightforward to check that the class of quantum automorphisms in the sense of Definition \ref{defqaut} is closed under taking tensor products and contragredients, 
and thus defines a full $ C^* $-tensor subcategory of the category of quantum permutations of the underlying vertex set.  
In particular, the collection of all finite dimensional quantum automorphisms yields a rigid $ C^* $-tensor category, which leads us to the following definition. 

\begin{definition} 
Let $ X = (V_X, E_X) $ be a graph. The quantum automorphism group $ \Qut_\delta(X) $ is the quantum subgroup of $ \Sym^+(V_X) $ corresponding to the rigid $ C^* $-tensor 
category of finite dimensional quantum automorphisms of $ X $. 
\end{definition} 

Here $ \Qut_\delta(X) $ is shorthand for \emph{discrete quantum automorphism group of} $ X $, emphasizing that this is not the same as the quantum 
automorphism group of $ X $ in the sense of Banica-Bichon if $ X $ is a finite graph. 

It is easy to check that a quantum automorphism of a graph $ X $ is the same thing as a quantum permutation $ \sigma = (\HH, u) $ 
of $ V_X $ such that $ A_X u = u A_X $ as matrices in $ M_{V_X}(B(\HH)) $. Note here that the entries of both matrix products in this formula 
make sense in the strong operator topology. 

The $ C^* $-algebra $ C_0(\Qut_\delta(X)) $ and its comultiplication can be described in analogy to the discussion in section \ref{secinfiniteqpg},  
and one verifies the following basic fact in the same way as Lemma \ref{onedim}.  

\begin{lemma} 
Unitary equivalence classes of one-dimensional quantum automorphisms of a graph $ X $ correspond bijectively to graph automorphisms of $ X $. 
\end{lemma} 

We also note that if $ X $ is a finite graph then $ \Qut_\delta(X) $ can be identified with the discretisation of the compact quantum automorphism group $ \Qut(X) $ discussed 
in section \ref{secprelim}. As already indicated above, this means in particular that $ C_0(\Qut_\delta(X)) $ is typically not isomorphic to $ C^\full(\Qut(X)) $.  

In the remainder of this section we shall discuss some examples regarding the existence and non-existence of quantum automorphisms of graphs. Following Banica and Bichon,
we say that a graph \emph{$ X $ has no quantum symmetry} if every irreducible quantum automorphism of $ X $ is one-dimensional. 
This is the case iff the entries $ u_{xy} $ of every quantum automorphism $ \sigma = (\HH, u) $ of $ X $ pairwise commute. 
Otherwise we say that \emph{$ X $ has quantum symmetry}. 

Let us first observe that this terminology is consistent with previous usage in the case of finite graphs. 

\begin{lemma} 
Let $ X $ be a finite graph. Then $ X $ has quantum symmetry in the sense above iff $ \Qut(X) $ is a non-classical compact quantum group. 
\end{lemma} 

\begin{proof} 
Note that $ \Qut(X) $ is non-classical iff $ C^\full(\Qut(X)) $ is nonabelian. By definition, existence of an irreducible quantum automorphism of $ X $ of dimension greater 
than one means that $ C^\full(\Qut(X)) $ is nonabelian. Conversely, if $ C^\full(\Qut(X)) $ is nonabelian then this $ C^* $-algebra admits an irreducible $ * $-representation 
of dimension greater than one, which means that $ X $ has quantum symmmetry.  
\end{proof}

\subsection{Infinite Johnson graphs} 

In his work on quantum symmetries, Schmidt has exhibited a number of criteria which allow one to check for the existence and non-existence of quantum 
automorphisms \cite{SCHMIDT_foldedcube}, \cite{SCHMIDT_distancetransitive}. Many of these criteria carry over to the case of infinite graphs in 
a straightforward way. 

Let us state one general result of this type in order to illustrate the situation. For vertices $ x,y $ in a graph $ X $ denote by $ d(x,y) $ the distance 
between $ x $ and $ y $, that is, the length of a shortest path connecting them. Here $ d(x,y) = \infty $ if there is no such path. 

\begin{lemma} \label{distancelemma}
Let $ X $ be a graph and let $ x_1, x_2, y_1, y_2 \in V_X $ such that $ d(x_1, x_2) \neq d(y_1, y_2) $. If $ (\HH, u) $ is a quantum automorphism of $ X $ 
then $ u_{x_1, y_1} u_{x_2 y_2} = 0 $. 
\end{lemma}

The proof of Lemma \ref{distancelemma} is the same as for Lemma 3.2 in \cite{SCHMIDT_distancetransitive}, simply note that all operators in the argument 
are uniformly bounded and that multiplication is jointly strongly continuous on bounded sets. 

In order to give a concrete example of an infinite graph without quantum symmetry let us consider infinite Johnson graphs. The \emph{Johnson graph} $ J(\infty,k) $ is the 
graph with vertices given by all $ k $-element subsets of $ \mathbb{N} $, such that two vertices are connected by an edge iff their intersection contains $ k - 1 $ elements, 
compare \cite{PANKOV_infinitejohnson}. The Johnson graph $ J(\infty,k) $ has diameter $ k $ and is distance transitive. 

\begin{prop} \label{johnson}
The Johnson graph $ J(\infty, 2) $ has no quantum symmetry. 
\end{prop} 

\begin{proof} 
The corresponding result for finite Johnson graphs is due to Schmidt, see Theorem 4.13 in \cite{SCHMIDT_distancetransitive}, and the argument given there carries over to the 
infinite case with minor modifications. This also involves appropriate versions of Lemma 3.4, Lemma 3.7 and Lemma 3.8 in \cite{SCHMIDT_distancetransitive}, which 
we shall however not spell out here. 

Let us fix a quantum automorphism $ (\HH,u) $ of $ J(\infty,2) $. Due to Lemma \ref{distancelemma} it suffices to show $ u_{ij} u_{kl} = u_{kl} u_{ij} $ for $ d(i,k) = d(j,l) = d $ 
for $ d = 1, 2 $ since $ J(\infty,2) $ has diameter $ 2 $. For $ d = 1 $ this means that both $ (i,k), (j,l) $ are edges, and since $ J(\infty,2) $ is distance transitive we may 
consider $ j = \{1,2\}, l = \{1,3\} $. 
The first four steps of the proof in \cite{SCHMIDT_distancetransitive} deal with the case $ d = 1 $, and the final step is devoted to the case $ d = 2 $. 

Keeping in mind the above remarks regarding results from section 3 of \cite{SCHMIDT_distancetransitive}, steps 1, 4 and 5 of this proof remain unchanged. 
For step 2, let $ p = \{1,a\} $ for $ a \geq 4 $. Then 
$$ 
u_{ij} u_{k, \{1,3\}} u_{ip} = 0 
$$ 
follows directly from 
$$
u_{ij} \bigg(u_{k \{2,a\}} + \sum_{c = 3, c \neq a}^\infty u_{k, \{1, c\}} \bigg) u_{ip} = 0 
$$
and the relation $ u_{ij} u_{k\{1,d\}} u_{ip} = u_{ij} u_{k \{1,3\}} u_{ip} $ for $ d \notin \{1,2, a\} $, by evaluating the terms on an arbitrary vector $ \xi \in \HH $. 
In the same way, step 3 is verified by noting that $ u_{ij}(u_{kl} + u_{k\{1,b\}}) u_{ip} = 0 $ for $ p = \{1,3\} $ and all $ b \geq 4 $ implies 
immediately $ u_{ij} u_{kl} u_{ip} = 0 $. 
\end{proof} 

The \emph{Kneser graph} $ K(\infty, 2) $ is the graph with vertex set given by all $ 2 $-element subsets of $ \mathbb{N} $ and an edge between two 
vertices iff the intersection of the corresponding sets is empty. In other words, $ K(\infty,2) $ is the complement of $ J(\infty,2) $. 
Since a graph has quantum symmetry iff its complement does, it follows from Proposition \ref{johnson} that $ K(\infty,2) $ has no quantum symmetry. 
We note that both $ J(\infty,2) $ and $ K(\infty,2) $ have infinite degree.

\subsection{Disjoint automorphism} 

Two automorphisms $ \sigma, \tau \in \Aut(X) $ of a graph $ X $ are called \emph{disjoint} if $ \sigma(x) \neq x \implies \tau(x) = x $ and $ \tau(x) \neq x \implies \sigma(x) = x $ 
for all $ x \in V_X $. The existence of a pair of disjoint automorphisms is sufficient for a finite graph $ X $ to have quantum symmetry, see 
Theorem 2.2 in \cite{SCHMIDT_foldedcube}. This criterion extends easily to the infinite setting. 

\begin{prop} \label{disjointautomorphisms}
Let $ X $ be a graph admitting a pair of disjoint automorphisms $ \sigma, \tau \in \Aut(X) $, and assume that $ k \in \mathbb{N} $ does not exceed the order
of neither of these automorphisms. Then $ X $ admits an irreducible quantum automorphism of dimension $ k $. In particular, if $ \Aut(X) $ contains a pair of nontrivial 
disjoint automorphisms then $ X $ has quantum symmetry. 
\end{prop} 

\begin{proof} 
We shall give the details for the sake of definiteness. 
Let us fix $ k \in \mathbb{N} $ such that $ k \leq \min(\ord(\sigma), \ord(\tau)) $, and note that if $ \ord(\sigma) = \infty = \ord(\tau) $ this means that we can 
choose $ k $ arbitrarily. We shall construct an irreducible quantum automorphism $ \rho = (\mathbb{C}^k, u^\rho) $ of dimension $ k $ as follows. 

Let $ G = \mathbb{Z}/k\mathbb{Z} $ and consider the actions of $ C(G) $ and $ C^*(G) $ on $ \mathbb{C}^k = l^2(G) $, induced by pointwise multiplication of functions in $ C(G) $, 
and the regular representation of $ G $, respectively. We shall write $ p_1, \dots p_k $ and $ q_1 \dots, q_k $ for the images in $ M_k(\mathbb{C}) = B(l^2(G)) $ of the minimal 
projections in $ C(G) $ and $ C^*(G) $ under these representations, respectively. 
 
Define $ u^\rho = (u^\rho_{xy})_{x,y \in V_X} $ by $ u^\rho = \sum_{r = 1}^k u^{\sigma^r} p_r + \sum_{s = 1}^k u^{\tau^s} q_s - u^{\id} 1 $. That 
is, the matrix $ u^\rho_{xy} \in M_k(\mathbb{C}) $ is given by 
$$ 
u^\rho_{xy} = \sum_{r = 1}^k \delta_{x \sigma^r(y)} p_r + \sum_{s = 1}^k \delta_{x \tau^s(y)} q_s - \delta_{xy} 1
$$ 
for $ x, y \in V_X $. Since $ \sigma $ and $ \tau $ are graph automorphisms we clearly have $ u^\rho A_X = A_X u^\rho $. Let 
\begin{align*}
M_{xy} &= \{1 \leq r \leq k \mid \sigma^r(y) = x \}, \\
N_{xy} &= \{1 \leq s \leq k \mid \tau^s(y) = x \}, 
\end{align*}
and observe that 
$$ 
u^\rho_{xy} = 
\begin{cases} 
\sum_{r \in M_{xy}} p_r & \text{ if } \sigma(y) \neq y  \\
\sum_{s \in N_{xy}} q_s & \text{ if } \tau(y) \neq y  \\
\delta_{xy} 1 & \text{ if } \sigma(y) = y = \tau(y). 
\end{cases}
$$
In particular, every $ u^\rho_{xy} $ for $ x, y \in V_X $ is a projection. In addition we have 
\begin{align*}
\sum_{x \in V_x} u^\rho_{xy} &= \sum_{r = 1}^k \sum_{x \in V_X} \delta_{x \sigma^r(y)} p_r + \sum_{s = 1}^k \sum_{x \in V_X} \delta_{x \tau^s(y)} q_s - 1
= \sum_{r = 1}^k  p_r + \sum_{s = 1}^k q_s - 1 = 1,  
\end{align*}
and similarly 
\begin{align*}
\sum_{y \in V_x} u^\rho_{xy} &= \sum_{r = 1}^k \sum_{y \in V_X} \delta_{x\sigma^r(y)} p_r + \sum_{s = 1}^k \sum_{y \in V_X} \delta_{x\tau^s(y)} q_s - 1 
= \sum_{r = 1}^k p_r + \sum_{s = 1}^k q_s - 1 = 1 
\end{align*}
as required. It follows that $ \rho = (\mathbb{C}^k, u^\rho) $ is a quantum automorphism of $ X $. 

Let us now check that $ \rho $ is irreducible. If $ \ord(\sigma) = m < \infty $ then 
upon decomposing $ X $ into the orbits of $ \sigma $ we see that there are elements $ v_1, \dots, v_a \in V_X $, fixed under $ \tau $, 
such that $ (\sigma^t(v_1), \dots, \sigma^t(v_a)) = (v_1, \dots, v_a) $ for $ 0 \leq t < k $ implies $ t = 0 $. 
If $ \ord(\sigma) = \infty $ we either find an infinite orbit, or orbits of arbitrarily large finite size. 
Again this allows us to choose $ v_1, \dots, v_a \in V_X $ such that $ (\sigma^t(v_1), \dots, \sigma^t(v_a)) = (v_1, \dots, v_a) $ for $ 0 \leq t < k $ implies $ t = 0 $.
In fact, we may take $ v_1 = \cdots = v_a $ for a suitably chosen vertex in this case. 
In either case it follows that $ u^\rho_{\sigma^r(v_1) v_1} \cdots u^\rho_{\sigma^r(v_a) v_a} = p_r $. 

In the same way we find $ w_1, \dots, w_b \in V_X $ such that $ u^\rho_{\tau^r(w_1) w_1} \cdots u^\rho_{\tau^r(w_b) w_b} = q_r $. 
Since the projections $ p_i, q_j $ for $ 1 \leq i,j \leq k $ generate $ M_k(\mathbb{C}) $ this yields the claim.
\end{proof} 

Proposition \ref{disjointautomorphisms} allows one to exhibit quantum symmetries in various concrete situations.

\subsection{Disjoint unions} 

In the case of finite graphs, the quantum automorphism group of a disjoint union has been determined by Bichon in terms of free wreath products. Using the constructions 
from section \ref{secfreewreath}, we shall now show that the same result holds in the infinite setting, in analogy to Theorem 4.2 in \cite{Bichonfreewreathproducts}. 

If $ (X_i)_{i \in I} $ is a collection of graphs labelled by some index set $ I $ we shall write $ X = \bigcup_{i \in I} X_i $ for their disjoint union, so 
that $ V_X = \bigcup_{i \in I} V_{X_i} $ and $ E_X = \bigcup_{i \in I}  E_{X_i} $. We will be interested in particular in the situation that all $ X_i $ are 
equal to a fixed graph. 

\begin{theorem} \label{disjointunion}
Let $ X $ be a connected graph and let $ I $ be an index set. Then there is a canonical isomorphism 
$$
\Qut_\delta\bigg(\bigcup_{i \in I} X\bigg) \cong \Qut_\delta(X) \Wr^* \Sym^+(I) 
$$
of discrete quantum groups. 
\end{theorem} 

\begin{proof} 
It suffices to construct a monoidal equivalence of the corresponding $ C^* $-tensor categories which leaves the underlying Hilbert spaces fixed. 

Assume first that $ \pi = (\pi^G, \pi^I) $ is a finite dimensional $ I $-wreath corepresentation of $ G = \Qut_\delta(X) $ with underlying Hilbert space $ \HH_\pi $. 
We construct a quantum permutation $ f(\pi) $ of the vertex set $ \bigcup_{i \in I} V_X $ on $ \HH_\pi $ by setting 
$$
f(\pi)_{x_i, y_j} = (\pi_i^G)_{xy} \pi^I_{ij} = \pi^I_{ij} (\pi_i^G)_{xy},  
$$
where $ x_i $ for $ x \in V_X $ and $ i \in I $ denotes the element $ x $ in the $ i $-th component of the disjoint union. 
Here we encode the representation $ \pi^G_i $ of $ C_0(\Qut_\delta(X)) $ by the corresponding quantum permutation. 
One checks that this is a quantum automorphism of $ \bigcup_{i \in I} X $. Indeed, assume that $ x_i, v_k $ are connected, which means $ i = k $ 
and $ (x,v) \in E_X $. If $ y_j, w_l $ are vertices with $ j \neq l $ we get $ f(\pi)_{x_i, y_j} f(\pi)_{v_k, w_l} = 0 $ since $ \pi^I_{ij} \pi^I_{il} = 0 $, 
and if $ j = l $ and $ (y,w) \notin E_X $ this follows from $ (\pi_i^G)_{xy} (\pi_i^G)_{vw} = 0 $. 

Clearly this construction defines a functor $ f: \Corep(\Qut_\delta(X) \Wr^* \Sym^+(I)) \rightarrow \Corep(\Qut_\delta(\bigcup_{i \in I} X)) $ 
which acts as the identity on morphisms. The trivial $ I $-wreath corepresentation of $ G $ is sent to the trivial quantum automorphism of $ \bigcup_{i \in I} X $. 
Moreover, if $ \pi, \eta $ are $ I $-free wreath corepresentations then 
\begin{align*}
f(\pi \tp \eta)_{x_i, y_j} &= ((\pi \tp \eta)_i^G)_{xy} (\pi \tp \eta)^I_{ij} \\
&= \sum_{v,k,l} (\pi^I_{ik} \otimes 1) ((\pi_i^G)_{xv} \otimes (\eta_k^G)_{vy}) (\pi^I_{il} \otimes \eta^I_{lj}) \\  
&= \sum_{v,k} ((\pi_i^G)_{xv} \otimes (\eta_k^G)_{vy}) (\pi^I_{ik} \otimes \eta^I_{kj}) \\  
&= \sum_{v,k} (\pi_i^G)_{xv} \pi^I_{ik} \otimes (\eta_k^G)_{vy} \eta^I_{kj} \\  
&= \sum_{v,k} f(\pi)_{x_i, v_k} \otimes f(\eta)_{v_k, y_j} \\  
&= (f(\pi) \tp f(\eta))_{x_i, y_j},  
\end{align*}
and it follows that $ f $ is in fact a unitary tensor functor. 

Conversely, assume that $ u = (u_{x_i, y_j}) $ is a finite dimensional quantum automorphism of $ \bigcup_{i \in I} X $. Define 
$$ 
g(u)^I_{ij} = \sum_v u_{x_i, v_j}
$$
where $ x $ is some fixed vertex in $ V_X $. Let us show that this is in fact independent of the choice of $ x $. Since $ X $ is connected 
it suffices to consider $ x, x' \in V_X $ with $ (x,x') \in E_X $ and compute 
\begin{align*}
\sum_v u_{x_i, v_j} &= \sum_{v,w,k} u_{x_i, v_j} u_{x'_i, w_k} \\
&= \sum_{v,w} u_{x_i, v_j} u_{x'_i, w_j} \\
&= \sum_{v,w,k} u_{x_i, v_k} u_{x'_i, w_j} \\
&= \sum_w u_{x'_i, w_j}, 
\end{align*}
using the defining relations for a quantum automorphism. Next we check that $ g(u)^I $ is a quantum permutation of $ I $. 
Since the projections $ u_{x_i, y_j} $ for fixed $ x_i $ are pairwise orthogonal it is clear that $ g(u)^I_{ij} $ is a projection. 
Moreover we have $ \sum_j g(u)^I_{ij} = 1 $ from the fact that $ u $ is a quantum permutation. 
Using our above considerations we obtain
$$
u_{w_i, y_j} g(u)^I_{ij} = \sum_v u_{w_i, y_j} u_{x_i, v_j} = \sum_v u_{w_i, y_j} u_{w_i, v_j} = u_{w_i, y_j} 
$$
for all $ w, y \in X $. In particular we get $ \sum_w u_{w_i, y_j} \leq \sum_v u_{x_i, v_j} $, and by symmetry we conclude that in fact 
$$
g(u)^I_{ij} = \sum_v u_{x_i, v_j} = \sum_w u_{w_i, y_j} 
$$
for any $ x, y \in V_X $. Using this observation it is straightforward to verify the remaining relation $ \sum_i g(u)^I_{ij} = 1 $. 

In addition we define 
$$
(g(u)^G_i)_{xy} = \sum_j u_{x_i, y_j}, 
$$
and we claim that $ g(u)^G_i $ is a quantum automorphism of $ X $. Clearly each $ (g(u)^G_i)_{xy} $ is a projection, and the relation $ \sum_y (g(u)^G_i)_{xy}  = 1 $ follows 
from the fact that $ u $ is a quantum permutation. From our above arguments we also obtain
\begin{align*}
\sum_x (g(u)^G_i)_{xy} = \sum_x \sum_j u_{x_i, y_j} = \sum_j \sum_x u_{x_i, y_j} = \sum_j \sum_w u_{v_i, w_j} = 1,  
\end{align*} 
where $ v \in V_X $ is arbitrary. Note here that $ u_{v_i, y_j} u_{w_i, y_k} = 0 $ for all $ v \neq w $ and $ j \neq k $ due to Lemma \ref{distancelemma}. 
For $ x,y,v,w \in V_X $ with $ (x,v) \in E_X $ and $ (y,w) \notin E_X $ we get 
$$
(g(u)^G_i)_{xy} (g(u)^G_i)_{vw} = \sum_{j,k} u_{x_i, y_j} u_{v_i, w_k} = 0 
$$
by definition of the adjacency relations in $ \bigcup_{i \in I} X $. Similarly, if $ (x,v) \notin E_X $ and $ (y,w) \in E_X $ we obtain 
$$
(g(u)^G_i)_{xy} (g(u)^G_i)_{vw} = \sum_{j,k} u_{x_i, y_j} u_{v_i, w_k} = \sum_{j \neq k} u_{x_i, y_j} u_{v_i, w_k} = 0  
$$
using Lemma \ref{distancelemma}, and we conclude that $ g(u)^G_i $ is indeed a quantum automorphism of $ X $. 

Moreover one has 
$$
(g(u)^G_i)_{xy} g(u)^I_{ij} = \sum_{v,k} u_{x_i, y_k} u_{x_i, v_j} = \sum_{v,k} u_{x_i, v_j} u_{x_i, y_k} = g(u)^I_{ij} (g(u)^G_i)_{xy} 
$$
for all $ x,y \in V_X $ and $ i,j \in I $. 
Setting $ g(u) = (g(u)^G, g(u)^X) $ we thus obtain an $ X $-wreath corepresentation of $ G = \Qut_\delta(X) $ on the underlying Hilbert space of $ u $. 
Clearly this construction yields a functor $ g: \Corep(\Qut_\delta(\bigcup_{i \in I} X)) \rightarrow \Corep(\Qut_\delta(X) \Wr^* \Sym^+(I))  $ which acts 
as the identity on morphisms. 

To see that $ f $ and $ g $ are mutually inverse equivalences of categories let $ \pi = (\pi^G, \pi^I) \in \Corep(\Qut_\delta(X) \Wr^* \Sym^+(I)) $ 
and compute 
\begin{align*}
((g f)(\pi)^G_i)_{xy} = \sum_j f(\pi)_{x_i, y_j} = \sum_j (\pi_i^G)_{xy} \pi^I_{ij} = (\pi_i^G)_{xy} 
\end{align*}  
and 
\begin{align*}
((g f)(\pi)^I)_{ij} = \sum_y f(\pi)_{x_i, y_j} = \sum_y (\pi_i^G)_{xy} \pi^I_{ij} = \pi^I_{ij}.  
\end{align*}  
Conversely, let $ u \in \Corep(\Qut_\delta(\bigcup_{i \in I} X)) $ and compute 
\begin{align*}
(fg)(u)_{x_i y_j} &= (g(u)_i^G)_{xy} g(u)^I_{ij} = \sum_{k,v} u_{x_i, y_k} u_{x_i, v_j} = u_{x_i y_j} 
\end{align*}
as required. Since both $ f $ and $ g $ act as the identity on morphism spaces this finishes the proof. 
\end{proof}

\subsection{Unit distance graphs} 

A \emph{unit distance graph} is a graph obtained by taking a subset of $ \mathbb{R}^d $ as vertex set and connecting two vertices iff their Euclidean distance is equal to $ 1 $. 
Examples of finite unit distance graphs in the plane include cycle graphs, hypercube graphs,
and the Petersen graph. Note that unit distance graphs in $ \mathbb{R}^2 $ need not be planar. It is known \cite{MAEHARA_ROEDL_unitdistance} that every finite graph can be represented 
as a unit distance graph for a suitable dimension $ d $. From our perspective, the class of unit distance graphs is interesting since it provides natural examples of graphs with uncountably many vertices. 

Consider for instance the unit distance graph $ U_d $ associated to Euclidean space $ \mathbb{R}^d $ itself. 
Here the case $ d = 1 $ is somewhat special, since $ U_d $ for $ d > 2 $ is connected while $ U_1 $ is highly disconnected. In fact, 
one can write $ U_1 \cong \bigcup_{x \in \mathbb{R}/\mathbb{Z}} L $ as disjoint union of uncountably many copies of the ``infinite line'' graph $ L $, that is, the unit 
distance graph of $ \mathbb{Z} \subset \mathbb{R} $. 

\begin{prop} \label{linegraph}
The ``infinite line'' graph $ L $ admits no quantum automorphisms which are not classical. 
\end{prop} 

\begin{proof} 
The following proof is due to Matthew Daws. As pointed out by Stefaan Vaes, an alternative argument can also be given using the techniques in \cite{ROLLIER_VAES_qut}. 

Assume that $ (\HH,u) $ is a quantum permutation of $ L $. From the adjacency relations we see that $ u_{i, j + 1} + u_{i, j - 1} = u_{i + 1,j} + u_{i - 1,j} $ 
for all $ i,j \in \mathbb{Z} $. We need to show that all projections $ u_{i,j} $ mutually commute. 

Since the problem is invariant under translations it suffices to check that $ u_{0,0} $ commutes with all other projections. 
From Lemma \ref{distancelemma} we know that $ u_{0,0} $ commutes with $ u_{i,j} $ provided $ |i| \neq |j| $. 
Therefore it is enough to verify that $ u_{0,0} $ commutes with all $ u_{i,j} $ such that $ |i| = |j| $. 

We prove this by induction on $ |i| = |j| = n $, the case $ n = 0 $ being trivial. By slight abuse of notation, we shall confuse projections in $ B(\HH) $ with 
their images in $ \HH $ in the sequel. We shall write $ \oplus $ to stress when the sum of two projections is orthogonal. 

Consider the case $ |i| = |j| = 1 $. Firstly, note that the Hilbert space $ \HH $ decomposes into an orthogonal direct sum
\begin{equation} \label{d1}
u_{1,1} \oplus u_{1,-1} \oplus (u_{1,1} + u_{1,-1})^\perp.
\end{equation}
As $ u_{-1,-1} \perp u_{1,-1} $ this implies 
\begin{equation} \label{d2}
u_{-1,-1} \leq u_{1,1} \oplus (u_{1,1} + u_{1,- 1})^\perp. 
\end{equation}
Using $  u_{-1, 1} + u_{-1, - 1} = u_{0, 0} + u_{-2,0} $ we get
\begin{equation} \label{d3}
u_{-1,-1} \leq u_{0, 0} + u_{-2,0}. 
\end{equation}
Moreover, by the magic unitary condition we know that $ u_{- 2, 0} $ is orthogonal to $ u_{0,0} + u_{2,0} = u_{1,1} + u_{1,-1} $, so 
\begin{equation} \label{d4}
u_{-2, 0} \leq (u_{1,1} + u_{1,-1})^\perp. 
\end{equation}
Combining \eqref{d3} and \eqref{d4} gives 
\begin{equation} \label{d5}
u_{-1, -1}  \leq u_{0,0} \oplus (u_{1,1} + u_{1,-1})^\perp.
\end{equation} 
From \eqref{d2} and \eqref{d5} we now get  
\begin{equation} \label{d6}
u_{-1,-1} \leq (u_{0,0} \wedge u_{1,1}) \oplus (u_{1,1} + u_{1,-1})^\perp, 
\end{equation}
where $ u_{0,0} \wedge u_{1,1} $ is the orthogonal projection onto the intersection of the images of $ u_{0,0} $ and $ u_{1,1} $. 
By symmetry we also have 
\begin{equation} \label{d7}
u_{-1,1} \leq (u_{0,0} \wedge u_{1,-1}) \oplus (u_{1,1} + u_{1,-1})^\perp.
\end{equation}
Combining \eqref{d6} and \eqref{d7} yields 
\begin{equation} \label{d8}
u_{-1, 1} \oplus u_{-1,-1} \leq (u_{0,0} \wedge u_{1,1}) \oplus (u_{0,0} \wedge u_{1,-1}) \oplus (u_{1,1} + u_{1,-1})^\perp. 
\end{equation}
Recall that $ u_{0,0} \leq u_{-1,1} \oplus u_{-1,-1} $ and $ u_{0,0} \leq u_{1,1} + u_{1,-1} $. Hence \eqref{d8} implies  
\begin{equation} \label{d10}
u_{0,0} \leq (u_{0,0} \wedge u_{1,1}) \oplus (u_{0,0} \wedge u_{1,-1}). 
\end{equation}
A symmetric argument gives 
\begin{equation} \label{d11} 
u_{2,0} \leq (u_{2,0} \wedge u_{1,1}) \oplus (u_{2,0} \wedge u_{1,-1}). 
\end{equation}
Combining \eqref{d10} and \eqref{d11} we obtain 
\begin{align*}
u_{1,1} \oplus u_{1,-1} &= u_{0,0} \oplus u_{2, 0} \\
&\leq (u_{0,0} \wedge u_{1,1}) \oplus (u_{0,0} \wedge u_{1,-1}) \oplus (u_{2,0} \wedge u_{1,1}) \oplus (u_{2,0} \wedge u_{1,-1}) \\
&\leq u_{1,1} \oplus u_{1,-1}. 
\end{align*}
It follows that we have equality throughout in this relation. In particular, $ u_{0,0} $ commutes with $ u_{1,1} $ and $ u_{1,-1} $. A symmetric 
argument shows that $ u_{0,0} $ commutes with $ u_{-1,1} $ and $ u_{-1,-1} $. 

Now assume that $ u_{0,0} $ commutes with $ u_{i,j} $ provided $ |i| = k = |j| $ for all $ 0 \leq k \leq n $ for some $ n \geq 1 $.  
Then $ u_{0,0} $ commutes with $ u_{n+1,n-1}, u_{n-1, n+1} $ and $ u_{n - 1, n-1} $. Using $ u_{n + 1, n + 1} + u_{n + 1, n - 1} = u_{n - 1, n - 1} + u_{n - 1, n + 1} $ 
it follows that $ u_{0,0} $ also commutes with $ u_{n + 1, n + 1} $. Again, by symmetry considerations we see that $ u_{0,0} $ 
commutes in fact with all $ u_{i,j} $ such that $ |i| = |j| = n + 1 $. This finishes the proof. 
\end{proof} 

Proposition \ref{linegraph} shows in particular that the canonical morphism $ \Aut(L) \rightarrow \Qut_\delta(L) $ is an isomorphism. 
Combining Proposition \ref{linegraph} with Theorem \ref{disjointunion} allows us to determine the quantum automorphism group of $ U_1 $ as follows. 

\begin{prop} 
The quantum automorphism group $ \Qut_\delta(U_1) $ is isomorphic to the free wreath product $ \Aut(L) \Wr^* \Sym^+(\mathbb{R}/\mathbb{Z}) $. 
\end{prop}  

In particular $ U_1 $ has quantum symmetry, and admits in fact uncountably many inequivalent quantum automorphisms in every dimension. 

For $ d > 1 $ the situation seems much less clear. According to the Beckman-Quarles theorem \cite{BECKMAN_QUARLES_isometries}, the classical 
symmetries of $ U_d $ for $ d > 1 $ are given precisely by the isometries of $ \mathbb{R}^d $. In particular, these graphs do not admit disjoint automorphisms. 
Let us pose the following question.  

\begin{question} 
Does $ U_d $ for $ d > 1 $ have quantum symmetry?
\end{question} 

Of course, unit distance graphs can be defined in an analogous fashion for subsets of more general metric spaces. It would be interesting to relate quantum 
symmetries of such graphs with the study of quantum automorphism groups of metric spaces, see \cite{GOSWAMI_existenceexamples}.

\subsection{Graph products} 

An easy way to obtain graphs with quantum symmetry out of known examples is to consider suitable products, see \cite{BANICA_BICHON_ordereleven}. 

Let us recall some definitions. If $ X, Y $ are graphs then the \emph{direct product} $ X \times Y $, also known as tensor product or Kronecker product, 
is the graph with vertex set $ V_{X \times Y} = V_X \times V_Y $ and adjacency matrix $ A_{X \times Y} = A_X \otimes A_Y $, so 
that $ (A_{X \times Y})_{(x_1, y_1),(x_2, y_2)} = (A_X)_{x_1, x_2} (A_Y)_{y_1, y_2} $. 
The \emph{cartesian product} $ X \Box Y $ is the graph with $ V_{X \Box Y} = V_X \times V_Y $ and $ A_{X \Box Y} = A_X \otimes 1 + 1 \otimes A_Y $, 
where $ 1 $ denotes the identity matrix.  
Finally, the \emph{strong product} $ X \boxtimes Y $ is given by $ V_{X \boxtimes Y} = V_X \times V_Y $ and $ A_{X \boxtimes Y} = (A_X + 1) \otimes (A_Y + 1) - 1 \otimes 1 $.  

\begin{prop} \label{graphproducts}
Let $ X, Y $ be graphs. Then there is a canonical morphism of discrete quantum groups $ \iota: \Qut_\delta(X) \times \Qut_\delta(Y) \rightarrow \Qut_\delta(X \times Y) $ such that 
the corresponding $ * $-homomorphism $ \iota^*: C_0(\Qut_\delta(X \times Y)) \rightarrow M(C_0(\Qut_\delta(X)) \otimes C_0(\Qut_\delta(Y))) $ satisfies   
$$
\iota^*(u_{(x_1, y_1), (x_2, y_2)}) = u_{x_1, x_2} \otimes u_{y_1, y_2} 
$$
for all $ x_1, x_2 \in V_X, y_1, y_2 \in V_Y $. The same holds if $ X \times Y $ is replaced by $ X \Box Y $ or $ X \boxtimes Y $.  
\end{prop}

\begin{proof} 
We consider only the case $ X \times Y $ since the arguments for $ X \Box Y $ and $ X \boxtimes Y $ are analogous. 

Assume that $ \sigma_X = (\HH, u^X), \sigma_Y = (\HH, u^Y) $ are quantum automorphisms of $ X $ and $ Y $ on the same Hilbert space $ \HH $
such that $ u^X_{x_1 x_2} u^Y_{y_1 y_2} = u^Y_{y_1 y_2} u^X_{x_1 x_2} $ for all $ x_1, x_2 \in V_X, y_1, y_2 \in V_Y $. Then 
$$
u^{X \times Y}_{(x_1, y_1), (x_2, y_2)} = u^X_{x_1 x_2} u^Y_{y_1 y_2} 
$$
determines a quantum  automorphism $ \sigma_{X \times Y} = (\HH, u^{X \times Y}) $ of the direct product $ X \times Y $. Indeed, 
the above formula defines a quantum permutation of $ V_X \times V_Y $ since the matrix elements $ u^X_{x_1 x_2} $ and $ u^Y_{y_1 y_2} $ commute, 
and we have
\begin{align*}
(u^{X \times Y} (A_X \otimes 1))_{(x_1, y_1), (x_2, y_2)} &= \sum_{v,w} u^{X \times Y}_{(x_1, y_1),(v,w)} (A_X)_{v,x_2} \delta_{w, y_2} \\
%&= \sum_{v,w} u^X_{x_1 v} u^Y_{y_1 w} (A_X)_{v x_2} \delta_{w,y_2} \\
&= \sum_{v} u^X_{x_1 v} (A_X)_{v,x_2} u^Y_{y_1 y_2} \\
&= \sum_{v} (A_X)_{x_1, v} u^X_{v x_2} u^Y_{y_1 y_2} \\
&= \sum_{v,w} (A_X)_{x_1,v} \delta_{y_1 w} u^{X \times Y}_{(v, w),(x_2,y_2)} \\
&= ((A_X \otimes 1)u^{X \times Y})_{(x_1, y_1), (x_2, y_2)}
\end{align*}
for all $ (x_1, y_1), (x_2, y_2) \in V_X \times V_Y $. A similar computation shows that $ u^{X \times Y} $ commutes with $ 1 \otimes A_Y $, and hence it also commutes 
with $ A_X \otimes A_Y $ as required. 

Applying this to $ \HH = \HH_\sigma \otimes \HH_\tau $ for $ \sigma = (\HH_\sigma, u^\sigma) \in \Qut_\delta(X), \tau = (\HH_\tau, u^\tau) \in \Qut_\delta(Y) $ 
and $ u^X = u^\sigma \otimes 1, u^Y = 1 \otimes u^\tau $ shows that we obtain a 
nondegenerate $ * $-homomorphism $ \iota^*: C_0(\Qut_\delta(X \times Y)) \rightarrow M(C_0(\Qut_\delta(X)) \otimes C_0(\Qut_\delta(Y))) $ 
such that $ \iota^*(u_{(x_1, y_1), (x_2, y_2)}) = u_{x_1, x_2} \otimes u_{y_1, y_2} $. It follows in particular that $ \iota^* $ is compatible with the 
comultiplications, and hence $ \iota^* $ implements a morphism of quantum groups $ \iota: \Qut_\delta(X) \times \Qut_\delta(Y) \rightarrow \Qut_\delta(X \times Y) $ as claimed. 
\end{proof} 

If one of the graphs $ X, Y $ has quantum symmetry then Proposition \ref{graphproducts} shows that the same is true for $ X \times Y, X \Box Y $ and $ X \boxtimes Y $. 
This allows one to produce basic examples of infinite graphs with quantum symmetry by taking products of finite graphs with quantum symmetry and arbitrary infinite graphs. 
Note that the graph products $ X \Box Y $ and $ X \boxtimes Y $ are connected whenever $ X $ and $ Y $ are. 

If $ X, Y $ are finite graphs then the sufficient criteria given in \cite{BANICA_BICHON_ordereleven} 
for having $ \Qut(X) \times \Qut(X) \cong \Qut(X \times Y) $ are also sufficient to show that 
the morphism $ \iota: \Qut_\delta(X) \times \Qut_\delta(X) \rightarrow \Qut_\delta(X \times Y) $ from Proposition \ref{graphproducts} is an isomorphism. 
It would be interesting to find natural criteria in the infinite situation. 

Let us also consider infinite products. If $ (X_i)_{i \in I} $ is an arbitrary family of graphs then the cartesian product $ X = \Box_{i \in I} X_i $ 
is defined by $ V_X = \prod_{i \in I} V_{X_i} $ and $ ((x_i),(y_i)) $ in $ E_X $ iff there exists $ l \in I $ such that $ (x_l, y_l) \in E_{X_l} $ and $ x_i = y_i $ 
for $ i \neq l $. Note that $ X $ has uncountably many vertices as soon as $ I $ is infinite and all $ X_i $ are nontrivial. 
In particular, the graph $ X $ is highly disconnected in this situation even if all $ X_i $ are connected.  

It is therefore customary to restrict attention to \emph{weak cartesian products}, see \cite{SABIDUSSI_graphmultiplication}. 
More precisely, if $ (X_i)_{i \in I} $ is a family of graphs and $ a = (a_i)_{i \in I} \in \prod_{i \in I} V_{X_i} $, then the weak cartesian product $ \Box_{i \in I}^a X_i $ 
is the induced subgraph of $ \Box_{i \in I} X_i $ corresponding to all $ x = (x_i)_{i \in I} $ such that $ x_i = a_i $ for all but finitely many $ i \in I $. 
If every $ X_i $ admits a transitive group of automorphisms then all weak cartesian products of $ (X_i)_{i \in I} $ are mutually isomorphic, 
and we can speak of \emph{the} weak cartesian product of the family of graphs. 

Let $ (G_i)_{i \in I} $ be a family of discrete quantum groups. We shall write $ \coprod_{i \in I} G_i $ for the discrete quantum group
$$
\coprod_{i \in I} G_i = \varinjlim_{F \subset I} \prod_{i \in F} G_i,  
$$
where the limit is taken over all finite subsets $ F \subset I $. 
This is a quantum analogue of the subgroup of an infinite product of groups consisting of those elements for which almost all entries are the identity. 

\begin{prop} \label{infinitecartesianqut}
Let $ (X_i)_{i \in I} $ be a family of graphs. Then there is a canonical embedding morphism of discrete quantum groups 
$$ 
\coprod_{i \in I} \Qut_\delta(X_i) \rightarrow \Qut_\delta(\Box_{i \in I}^a X_i) 
$$
for any $ a = (a_i)_{i \in I} \in \prod_{i \in I} V_{X_i} $. 
\end{prop} 

\begin{proof} 
We define a fully faithful monoidal functor $ \iota_F: \Corep(\prod_{i \in F} \Qut_\delta(X_i)) \rightarrow \Corep(\Qut_\delta(\Box_{i \in I}^a X_i)) $ 
for a finite subset $ F \subset I $ as follows. 

Let $ (\HH_j, u^j)_{j \in F} $ be a family of finite dimensional quantum automorphisms of the graphs $ X_j $. Then we obtain a quantum 
automorphism $ \iota_F((\HH_j, u^j)_{j \in F}) $ of $ \Box_{i \in I}^a X_i $ on $ \bigotimes_{i \in F} \HH_i $ by setting 
$$
\iota_F((\HH_j, u^j)_{j \in F})_{(x_i), (y_i)} = \prod_{k \in I \setminus F} \delta_{x_k, y_k} \bigotimes_{k \in F} u^k_{x_k, y_k}
$$
for vertices $ (x_i), (y_i) $ in $ \Box_{i \in I}^a X_i $. Note here that $ x_i = a_i = y_i $ for all but finitely many $ i \in I $. 
If all $ (\HH_j, u^j) $ are irreducible then the same is true for $ \iota_F((\HH_j, u^j)_{j \in F}) $, 
and one checks that $ \iota_F $ is compatible with tensor products. Since every irreducible object in $ \Corep(\prod_{i \in F} \Qut_\delta(X_i)) $ is 
isomorphic to the tensor product of a family of irreducible finite dimensional quantum automorphisms $ (\HH_j, u^j)_{j \in F} $ as above 
it follows that $ \iota_F $ induces a fully faithful monoidal functor as required. 

Since the functors $ \iota_F $ for varying subsets $ F \subset I $ are clearly compatible with inclusions this yields the claim. 
\end{proof} 

In general the morphism from Proposition \ref{infinitecartesianqut} is not an isomorphism. 
It would be interesting to find suitable conditions under which one gets in fact an isomorphism this way.

\subsection{Infinite Hamming graphs} 

Fix $ n \in \mathbb{N} $ and consider the complete graph $ X_k = K_n $ for all $ k \in \mathbb{N} $, which we assume to be modelled on the set $ \mathbb{Z}/n\mathbb{Z} $. 
Fixing the connected component of the vertex $ a = 0 \in \prod \mathbb{Z}/n\mathbb{Z} $ we see that the associated weak cartesian product $ X = \Box_{k \in \mathbb{N}}^a K_n $ has 
vertex set $ V_X = \bigoplus_{k \in \mathbb{N}} \mathbb{Z}/n\mathbb{Z} $. By definition, two vertices $ x, y \in V_X $ are connected iff there exists $ l \in \mathbb{N} $ 
such that $ x_k = y_k $ for all $ k \neq l $. The graph $ X = H(\infty, n) $ is an \emph{infinite Hamming graph}. 

Recall from section \ref{secfreewreath} the construction of wreath products for discrete quantum groups. 

\begin{prop} \label{hamming}
The quantum automorphism group $ \Qut_\delta(H(\infty,n)) $ contains the restricted wreath product $ \Sigma_n^+ \wwr \Sym(\mathbb{N}) $ naturally as a quantum subgroup. 
\end{prop} 

\begin{proof} 
Due to Proposition \ref{infinitecartesianqut} we obtain a canonical embedding of $ \coprod_{k \in \mathbb{N}} \Sigma_n^+ $ into $ \Qut_\delta(H(\infty,n)) $. Together with 
the natural action of $ \Sym(\mathbb{N}) $ given by permutation of factors this yields an embedding $ \Sigma_n^+ \wwr \Sym(\mathbb{N}) \rightarrow \Qut_\delta(H(\infty,n)) $ as 
required. Since the verifications are analogous to the arguments in the proof of Proposition \ref{infinitecartesianqut} we shall not spell out the details. 
\end{proof} 

The quantum automorphism groups of finite Hamming graphs $ \Box_{k = 1}^m K_n $ have recently been computed by Gromada \cite{GROMADA_quantumsymmetriescayley}. 
We note that the inclusion morphism in Proposition \ref{hamming} is not an isomorphism, and we shall leave it as an open problem to determine $ \Qut_\delta(H(\infty,n)) $. 
It would also be interesting to understand the structure of the quantum automorphism groups of other weak cartesian products.

\subsection{The Rado graph} 

Let us finally consider the Rado graph, also known as the Erd\H{o}s-R\'enyi graph or random graph, see \cite{ERDOES_RENYI_asymmetricgraphs}, \cite{RADO_universalgraphs}, 
\cite{CAMERON_randomgraphrevisited}. It can be defined as the countable graph $ R $ with vertex set $ V_R $ given by the prime numbers congruent $ 1 \bmod 4 $, 
and with $ (p,q) \in E_R $ iff $ p $ is a quadratic residue mod $ q $. Note here that $ (p,q) \in E_R $ iff $ (q,p) \in E_R $ by the law of quadratic reciprocity. 

The Rado graph admits numerous concrete models, and contains all countable graphs as induced subgraphs. 
A key property of this graph is that for any pair of disjoint finite sets $ A,B \subset V_R $ there exists a vertex $ w \in V_R $ such that $ (x,w) \in E_R $ 
for all $ x \in A $ and $ (y,w) \notin R $ for all $ y \in B $. 
In fact, the Rado graph $ R $ is homogeneous \cite{HENSON_homogeneousgraphs}, which means that any partial automorphism of $ R $ 
can be extended to a global automorphism. 

As a consequence, this graph has a wealth of automorphisms, which can be constructed by the back-and-forth method. 
In contrast, we have the following result regarding quantum automorphisms. 

\begin{prop} \label{ER}
The Rado graph $ R $ admits no finite dimensional quantum automorphisms which are not classical. 
\end{prop} 

\begin{proof} 
Assume that $ \sigma = (\HH, u) $ is a non-classical finite dimensional quantum automorphism, by which we mean that the algebra generated by the projections $ u_{xy} $ 
for $ x, y \in V_R $ is noncommutative. 

By assumption, we then find vertices $ x(+), x(-) $ and $ y_1(+), y_1(-) $ such $ u_{x(+), y_1(+)} $ and $ u_{x(-), y_1(-)} $ do not commute. 
We let $ u_{x(\pm), y_j(\pm)} $ for $ j = 2, \dots, r_\pm $ be the remaining nonzero projections in the row for $ x(\pm) $. Set 
$$
p_\alpha(\pm) = u_{x(\pm) y_1(\pm)}, \qquad p_\beta(\pm) = \sum_{j > 1} u_{x(\pm)y_j(\pm)}.  
$$
Then the projections $ p_\alpha(+), p_\beta(+) $ are orthogonal with $ p_\alpha(+) + p_\beta(+) = 1 $, 
and in the same way $ p_\alpha(-), p_\beta(-) $ are orthogonal with $ p_\alpha(-) + p_\beta(-) = 1 $. 
By construction, $ p_\alpha(+) $ and $ p_\alpha(-) $ do not commute, and hence $ p_\beta(+) $ and $ p_\beta(-) $ do not commute either. 

Choose a vertex $ w $ such that $ (w, y_1(+)) \in E_R, (w,y_1(-)) \in E_R $ and $ (w, y_i(\pm)) \notin E_R $ for all $ i > 1 $,
and let us consider the nonzero projections $ u_{v_i w} $ in the column for $ w $, where $ 1 \leq i \leq k $ for some $ k $. 
If $ (v_i, x(\pm)) \in E_R $ then $ u_{v_i w} $ is orthogonal to $ p_\beta(\pm) $, 
and if $ (v_i, x(\pm)) \notin E_R $ then $ u_{v_i w} $ is orthogonal to $ p_\alpha(\pm) $. 
We conclude $ u_{v_iw} \leq p_\alpha(\pm) $ in the first case and $ u_{v_i w} \leq p_\beta(\pm) $ in the second case. 
Since the projections $ u_{v_i w} $ form a partition of unity it follows that we can write each of $ p_\alpha(\pm) $ and $ p_\beta(\pm) $ as sums of certain $ u_{v_i w} $. 

However, this means in particular that all these projections commute, which yields a contradiction. 
\end{proof} 

According to Proposition \ref{ER}, any non-classical quantum automorphism of $ R $ is necessarily infinite dimensional. 
It follows from Proposition \ref{ER} and Proposition \ref{disjointautomorphisms} that $ \Aut(R) $ does not contain disjoint automorphisms, and we pose the following question. 

\begin{question} 
Does $ R $ have quantum symmetry?
\end{question} 

We note that the argument in Proposition \ref{ER} works in the same way for the Henson graphs $ G_p $ for $ p \geq 4 $, 
that is, the universal $ K_p $-free graphs. In other words, none of these graphs admits finite dimensional quantum automorphisms. 
This follows from Lemma 2.1 in \cite{HENSON_homogeneousgraphs}.

\bibliographystyle{hacm}

\bibliography{cvoigt}

\def\cprime{$'$} \def\cprime{$'$} \def\cprime{$'$} \def\cprime{$'$}
  \def\cprime{$'$} \def\cprime{$'$}
  \def\polhk#1{\setbox0=\hbox{#1}{\ooalign{\hidewidth
  \lower1.5ex\hbox{`}\hidewidth\crcr\unhbox0}}} \def\cprime{$'$}
  \def\cprime{$'$} \def\cprime{$'$} \def\Dbar{\leavevmode\lower.6ex\hbox to
  0pt{\hskip-.23ex \accent"16\hss}D}
  \def\cftil#1{\ifmmode\setbox7\hbox{$\accent"5E#1$}\else
  \setbox7\hbox{\accent"5E#1}\penalty 10000\relax\fi\raise 1\ht7
  \hbox{\lower1.15ex\hbox to 1\wd7{\hss\accent"7E\hss}}\penalty 10000
  \hskip-1\wd7\penalty 10000\box7}
  \def\cfudot#1{\ifmmode\setbox7\hbox{$\accent"5E#1$}\else
  \setbox7\hbox{\accent"5E#1}\penalty 10000\relax\fi\raise 1\ht7
  \hbox{\raise.1ex\hbox to 1\wd7{\hss.\hss}}\penalty 10000 \hskip-1\wd7\penalty
  10000\box7}
\begin{thebibliography}{10}

\bibitem{BANICA_quantumpermutationgroups}
{\sc Banica, T.}
\newblock Quantum permutation groups.
\newblock arXiv 2012.10975.

\bibitem{Banicageneric}
{\sc Banica, T.}
\newblock Symmetries of a generic coaction.
\newblock {\em Math. Ann. 314}, 4 (1999), 763--780.

\bibitem{Banicafusscatalan}
{\sc Banica, T.}
\newblock Quantum groups and {F}uss-{C}atalan algebras.
\newblock {\em Comm. Math. Phys. 226}, 1 (2002), 221--232.

\bibitem{Banicaqutgraph}
{\sc Banica, T.}
\newblock Quantum automorphism groups of homogeneous graphs.
\newblock {\em J. Funct. Anal. 224}, 2 (2005), 243--280.

\bibitem{Banicasmallmetric}
{\sc Banica, T.}
\newblock Quantum automorphism groups of small metric spaces.
\newblock {\em Pacific J. Math. 219}, 1 (2005), 27--51.

\bibitem{BANICA_BICHON_ordereleven}
{\sc Banica, T., and Bichon, J.}
\newblock Quantum automorphism groups of vertex-transitive graphs of order
  {$\leq11$}.
\newblock {\em J. Algebraic Combin. 26}, 1 (2007), 83--105.

\bibitem{BBfourpoints}
{\sc Banica, T., and Bichon, J.}
\newblock Quantum groups acting on 4 points.
\newblock {\em J. Reine Angew. Math. 626\/} (2009), 75--114.

\bibitem{BBCsurvey}
{\sc Banica, T., Bichon, J., and Collins, B.}
\newblock Quantum permutation groups: a survey.
\newblock In {\em Noncommutative harmonic analysis with applications to
  probability}, vol.~78 of {\em Banach Center Publ.} Polish Acad. Sci. Inst.
  Math., Warsaw, 2007, pp.~13--34.

\bibitem{BBpauli}
{\sc Banica, T., and Collins, B.}
\newblock Integration over the {P}auli quantum group.
\newblock {\em J. Geom. Phys. 58}, 8 (2008), 942--961.

\bibitem{BECKMAN_QUARLES_isometries}
{\sc Beckman, F.~S., and Quarles, Jr., D.~A.}
\newblock On isometries of {E}uclidean spaces.
\newblock {\em Proc. Amer. Math. Soc. 4\/} (1953), 810--815.

\bibitem{BMTcqgcoamenability}
{\sc B{\'e}dos, E., Murphy, G.~J., and Tuset, L.}
\newblock Co-amenability of compact quantum groups.
\newblock {\em J. Geom. Phys. 40}, 2 (2001), 130--153.

\bibitem{BICHON_qutgraphs}
{\sc Bichon, J.}
\newblock Quantum automorphism groups of finite graphs.
\newblock {\em Proc. Amer. Math. Soc. 131}, 3 (2003), 665--673.

\bibitem{Bichonfreewreathproducts}
{\sc Bichon, J.}
\newblock Free wreath product by the quantum permutation group.
\newblock {\em Algebr. Represent. Theory 7}, 4 (2004), 343--362.

\bibitem{Brannanquantumautomorphism}
{\sc Brannan, M.}
\newblock Reduced operator algebras of trace-perserving quantum automorphism
  groups.
\newblock {\em Doc. Math. 18\/} (2013), 1349--1402.

\bibitem{BRANNAN_CHIRVASITU_FRESLON_matrixmodels}
{\sc Brannan, M., Chirvasitu, A., and Freslon, A.}
\newblock Topological generation and matrix models for quantum reflection
  groups.
\newblock {\em Adv. Math. 363\/} (2020), 106982, 31.

\bibitem{CAMERON_randomgraphrevisited}
{\sc Cameron, P.~J.}
\newblock The random graph revisited.
\newblock In {\em European {C}ongress of {M}athematics, {V}ol. {I}
  ({B}arcelona, 2000)}, vol.~201 of {\em Progr. Math.} Birkh\"{a}user, Basel,
  2001, pp.~267--274.

\bibitem{DFSWhaagerup}
{\sc Daws, M., Fima, P., Skalski, A., and White, S.}
\newblock The {H}aagerup property for locally compact quantum groups.
\newblock {\em J. Reine Angew. Math. 711\/} (2016), 189--229.

\bibitem{ERDOES_RENYI_asymmetricgraphs}
{\sc Erd\H{o}s, P., and R\'{e}nyi, A.}
\newblock Asymmetric graphs.
\newblock {\em Acta Math. Acad. Sci. Hungar. 14\/} (1963), 295--315.

\bibitem{FIMA_propertyt}
{\sc Fima, P.}
\newblock Kazhdan's property {$T$} for discrete quantum groups.
\newblock {\em Internat. J. Math. 21}, 1 (2010), 47--65.

\bibitem{Freslonpermanence}
{\sc Freslon, A.}
\newblock Permanence of approximation properties for discrete quantum groups.
\newblock {\em Ann. Inst. Fourier (Grenoble) 65}, 4 (2015), 1437--1467.

\bibitem{GOSWAMI_existenceexamples}
{\sc Goswami, D.}
\newblock Existence and examples of quantum isometry groups for a class of
  compact metric spaces.
\newblock {\em Adv. Math. 280\/} (2015), 340--359.

\bibitem{GOSWAMI_SKALSKI_quantumpermutations}
{\sc Goswami, D., and Skalski, A.}
\newblock On two possible constructions of the quantum semigroup of all quantum
  permutations of an infinite countable set.
\newblock In {\em Operator algebras and quantum groups}, vol.~98 of {\em Banach
  Center Publ.} Polish Acad. Sci. Inst. Math., Warsaw, 2012, pp.~199--214.

\bibitem{GROMADA_quantumsymmetriescayley}
{\sc Gromada, D.}
\newblock Quantum symmetries of {C}ayley graphs of abelian groups.
\newblock arXiv:2106.08787.

\bibitem{HMPSsynchronousgames}
{\sc Helton, J.~W., Meyer, K.~P., Paulsen, V.~I., and Satriano, M.}
\newblock Algebras, synchronous games, and chromatic numbers of graphs.
\newblock {\em New York J. Math. 25\/} (2019), 328--361.

\bibitem{HENSON_homogeneousgraphs}
{\sc Henson, C.~W.}
\newblock A family of countable homogeneous graphs.
\newblock {\em Pacific J. Math. 38\/} (1971), 69--83.

\bibitem{KOESTLER_SPEICHER_definetti}
{\sc K\"{o}stler, C., and Speicher, R.}
\newblock A noncommutative de {F}inetti theorem: invariance under quantum
  permutations is equivalent to freeness with amalgamation.
\newblock {\em Comm. Math. Phys. 291}, 2 (2009), 473--490.

\bibitem{KUSTERMANS_kmsweights}
{\sc Kustermans, J.}
\newblock {KMS}-weights on ${C}^*$-algebras.
\newblock arXiv:funct-an/9704008.

\bibitem{KVLCQG}
{\sc Kustermans, J., and Vaes, S.}
\newblock Locally compact quantum groups.
\newblock {\em Ann. Sci. \'Ecole Norm. Sup. (4) 33}, 6 (2000), 837--934.

\bibitem{KYED_betticoamenable}
{\sc Kyed, D.}
\newblock {$L^2$}-{B}etti numbers of coamenable quantum groups.
\newblock {\em M\"{u}nster J. Math. 1\/} (2008), 143--179.

\bibitem{LUPINI_MANCINSKA_ROBERSON_nonlocalgames}
{\sc Lupini, M., Man\v{c}inska, L., and Roberson, D.~E.}
\newblock Nonlocal games and quantum permutation groups.
\newblock {\em J. Funct. Anal. 279}, 5 (2020), 108592, 44.

\bibitem{MAEHARA_ROEDL_unitdistance}
{\sc Maehara, H., and R\"{o}dl, V.}
\newblock On the dimension to represent a graph by a unit distance graph.
\newblock {\em Graphs Combin. 6}, 4 (1990), 365--367.

\bibitem{MANCINSKA_ROBERSON_quantumisoplanar}
{\sc Man\v{c}inska, L., and Roberson, D.}
\newblock Quantum isomorphism is equivalent to equality of homomorphism counts
  from planar graphs.
\newblock In {\em 2020 {IEEE} 61st {A}nnual {S}ymposium on {F}oundations of
  {C}omputer {S}cience}. IEEE Computer Soc., Los Alamitos, CA, [2020]
  \copyright 2020, pp.~661--672.

\bibitem{MARKER_modeltheory}
{\sc Marker, D.}
\newblock {\em Model theory}, vol.~217 of {\em Graduate Texts in Mathematics}.
\newblock Springer-Verlag, New York, 2002.
\newblock An introduction.

\bibitem{NTlecturenotes}
{\sc Neshveyev, S., and Tuset, L.}
\newblock {\em Compact quantum groups and their representation categories},
  vol.~20 of {\em Cours Sp\'{e}cialis\'{e}s [Specialized Courses]}.
\newblock Soci\'{e}t\'{e} Math\'{e}matique de France, Paris, 2013.

\bibitem{PANKOV_infinitejohnson}
{\sc Pankov, M.}
\newblock Automorphisms of infinite {J}ohnson graphs.
\newblock {\em Discrete Math. 313}, 5 (2013), 721--725.

\bibitem{RADO_universalgraphs}
{\sc Rado, R.}
\newblock Universal graphs and universal functions.
\newblock {\em Acta Arith. 9\/} (1964), 331--340.

\bibitem{ROLLIER_VAES_qut}
{\sc Rollier, L., and Vaes, S.}
\newblock Quantum automorphism groups of connected locally finite graphs and
  quantizations of discrete groups.
\newblock {\em arXiv:2209.03770\/} (2022).

\bibitem{SABIDUSSI_graphmultiplication}
{\sc Sabidussi, G.}
\newblock Graph multiplication.
\newblock {\em Math. Z. 72\/} (1959/60), 446--457.

\bibitem{SCHMIDT_distancetransitive}
{\sc Schmidt, S.}
\newblock On the quantum symmetry of distance-transitive graphs.
\newblock {\em Adv. Math. 368\/} (2020), 107150, 50.

\bibitem{SCHMIDT_foldedcube}
{\sc Schmidt, S.}
\newblock Quantum automorphisms of folded cube graphs.
\newblock {\em Ann. Inst. Fourier (Grenoble) 70}, 3 (2020), 949--970.

\bibitem{SOLTAN_bohr}
{\sc So{\l}tan, P.~M.}
\newblock Quantum {B}ohr compactification.
\newblock {\em Illinois J. Math. 49}, 4 (2005), 1245--1270.

\bibitem{SOLTAN_compactifications}
{\sc So{\l}tan, P.~M.}
\newblock Compactifications of discrete quantum groups.
\newblock {\em Algebr. Represent. Theory 9}, 6 (2006), 581--591.

\bibitem{Tomatsuamenablediscrete}
{\sc Tomatsu, R.}
\newblock Amenable discrete quantum groups.
\newblock {\em J. Math. Soc. Japan 58}, 4 (2006), 949--964.

\bibitem{vDmult}
{\sc Van~Daele, A.}
\newblock Multiplier {H}opf algebras.
\newblock {\em Trans. Amer. Math. Soc. 342}, 2 (1994), 917--932.

\bibitem{Wangqsymmetry}
{\sc Wang, S.}
\newblock Quantum symmetry groups of finite spaces.
\newblock {\em Comm. Math. Phys. 195}, 1 (1998), 195--211.

\bibitem{Woronowiczsuqn}
{\sc Woronowicz, S.~L.}
\newblock Tannaka-{K}re\u\i n duality for compact matrix pseudogroups.
  {T}wisted {${\rm SU}(N)$} groups.
\newblock {\em Invent. Math. 93}, 1 (1988), 35--76.

\end{thebibliography}

\end{document}